\documentclass[11pt]{article}
\usepackage[colorlinks=true,bookmarks=false,linkcolor=blue,urlcolor=blue,citecolor=blue,breaklinks=true]{hyperref}
\usepackage{breakcites}

\usepackage[ansinew]{inputenc}

\usepackage{doi}

\usepackage{amsmath}

\usepackage{mathtools}

\usepackage{amsfonts}
\usepackage{amsthm}
\usepackage{amssymb}

\usepackage{color}
\usepackage{bm} 
\usepackage{dsfont} 

\usepackage{graphicx}
\usepackage[lmargin=3cm,rmargin=3cm,bottom=3cm,top=3cm]{geometry}

\usepackage{tikz}
\definecolor{mycolor1}{rgb}{0.105882,0.619608,0.466667}
\definecolor{mycolor2}{rgb}{0.85098,0.372549,0.00784314}
\definecolor{mycolor3}{rgb}{0.458824,0.439216,0.701961}
\definecolor{mycolor4}{rgb}{0.905882,0.160784,0.541176}
\definecolor{mycolor5}{rgb}{0.4,0.65098,0.117647}
\definecolor{mycolor6}{rgb}{0.65098,0.462745,0.113725}
\definecolor{mycolor7}{rgb}{0.901961,0.670588,0.00784314}
\definecolor{mycolor8}{rgb}{0.4,0.4,0.4}
\definecolor{mycolor9}{rgb}{0.301961,0,0.294118}
\definecolor{mycolor10}{rgb}{0.0313725,0.25098,0.505882}

\usetikzlibrary{calc}        

\makeatletter
\newif\ifmygrid@coordinates
\tikzset{/mygrid/step line/.style={line width=0.80pt,draw=gray!80},
         /mygrid/steplet line/.style={line width=0.25pt,draw=gray!80}}
\pgfkeys{/mygrid/.cd,
         step/.store in=\mygrid@step,
         steplet/.store in=\mygrid@steplet,
         coordinates/.is if=mygrid@coordinates}
\def\mygrid@def@coordinates(#1,#2)(#3,#4){%
    \def\mygrid@xlo{#1}%
    \def\mygrid@xhi{#3}%
    \def\mygrid@ylo{#2}%
    \def\mygrid@yhi{#4}%
}
\newcommand\DrawGrid[3][]{%
    \pgfkeys{/mygrid/.cd,coordinates=true,step=1,steplet=0.2,#1}%
    \draw[/mygrid/steplet line] #2 grid[step=\mygrid@steplet] #3;
    \draw[/mygrid/step line] #2 grid[step=\mygrid@step] #3;
    \mygrid@def@coordinates#2#3%
    \ifmygrid@coordinates%
        \draw[/mygrid/step line]
        \foreach \xpos in {\mygrid@xlo,...,\mygrid@xhi} {%
          (\xpos,\mygrid@ylo) -- ++(0,-3pt)
                              node[anchor=north] {$\xpos$}
        }
        \foreach \ypos in {\mygrid@ylo,...,\mygrid@yhi} {%
          (\mygrid@xlo,\ypos) -- ++(-3pt,0)
                              node[anchor=east] {$\ypos$}
        };
    \fi%
}
\makeatother

\usepackage{psfrag}

\usepackage{algorithm}
\usepackage{algpseudocode}
\algdef{SE}[DOWHILE]{Do}{doWhile}{\algorithmicdo}[1]{\algorithmicwhile\ #1}%

\usepackage{sectsty}
\makeatletter\def\@seccntformat#1{\protect\makebox[0pt][r]{\csname the#1\endcsname\hspace{12pt}}}\makeatother

\usepackage{mathrsfs} 

\usepackage[round,compress]{natbib}

\newcommand{\remove}[1]{}
\newcommand{\removesafe}[1]{}

\usepackage{subfigure}

\newcommand{\transpose}{^\top\! }

\newcommand{\inner}[2]{\left\langle{#1},{#2}\right\rangle}
\newcommand{\innersmall}[2]{\langle{#1},{#2}\rangle}

\newcommand{\trace}{\mathrm{Tr}}
\newcommand{\Trace}{\mathrm{Tr}}

\newcommand{\Proj}{\mathrm{Proj}}

\newcommand{\Exp}{\mathrm{Exp}}

\newcommand{\Retr}{\mathrm{Retr}}

\newcommand{\T}{\mathrm{T}}

\newcommand{\Rnn}{{\mathbb{R}^{n\times n}}}
\newcommand{\Rnp}{{\mathbb{R}^{n\times p}}}

\newcommand{\Rp}{{\mathbb{R}^{p}}}

\newcommand{\reals}{{\mathbb{R}}}

\newcommand{\Rn}{{\mathbb{R}^n}}
\newcommand{\Rm}{{\mathbb{R}^m}}

\newcommand{\grad}{\mathrm{grad}}
\newcommand{\Hess}{\mathrm{Hess}}

\newcommand{\diag}{\mathrm{diag}}
\newcommand{\D}{\mathrm{D}}
\newcommand{\dt}{\mathrm{d}t}
\newcommand{\ddt}{\frac{\mathrm{d}}{\mathrm{d}t}}
\newcommand{\dds}{\frac{\mathrm{d}}{\mathrm{d}s}}
\newcommand{\ds}{\mathrm{d}s}

\newcommand{\calE}{\mathcal{E}}

\newcommand{\calO}{\mathcal{O}}
\newcommand{\calM}{\mathcal{M}}

\newcommand{\Id}{\operatorname{Id}} 

\newcommand{\opnorm}[1]{\left\|{#1}\right\|_\mathrm{2}}

\newcommand{\frobnormsmall}[1]{\|{#1}\|_\mathrm{F}}

\newcommand{\lambdamin}{\lambda_\mathrm{min}}

\newcommand{\dist}{\mathrm{dist}}

\newcommand{\II}{I\!I}

\newtheorem{theorem}{Theorem} 
\newtheorem{lemma}[theorem]{Lemma}
\newtheorem{proposition}[theorem]{Proposition}
\newtheorem{corollary}[theorem]{Corollary}
\newtheorem{assumption}{A\ignorespaces} 
\newtheorem{definition}{Definition} 
\newtheorem{remark}[theorem]{Remark}

\usepackage{enumitem} 
\usepackage[normalem]{ulem} 
\newcommand{\aref}[1]{\hyperref[#1]{A\ref{#1}}}

\newcommand{\In}{I}

\renewcommand{\L}{L_g}
\newcommand{\Lg}{L_g}
\newcommand{\LH}{L_H}

\title{Global rates of convergence\\for nonconvex optimization on manifolds}

\author{Nicolas Boumal\thanks{Mathematics Department and PACM, Princeton University, Princeton, NJ, USA.} \and P.-A.\ Absil\thanks{ICTEAM Institute, Universit\'e catholique de Louvain, Louvain-la-Neuve, Belgium.} \and Coralia Cartis\thanks{Mathematical Institute, University of Oxford, Oxford, UK.}}

\begin{document}

\maketitle

\begin{abstract}
	
We consider the minimization of a cost function $f$ on a manifold $\calM$ using Riemannian gradient descent and Riemannian trust regions (RTR). We focus on satisfying necessary optimality conditions within a tolerance $\varepsilon$. Specifically, we show that, under Lipschitz-type assumptions on the pullbacks of $f$ to the tangent spaces of~$\calM$, both of these algorithms produce points with Riemannian gradient smaller than $\varepsilon$ in $\mathcal{O}(1/\varepsilon^2)$ iterations. Furthermore, RTR returns a point where also the Riemannian Hessian's least eigenvalue is larger than $-\varepsilon$ in $\mathcal{O}(1/\varepsilon^3)$ iterations. There are no assumptions on initialization. The rates match their (sharp) unconstrained counterparts as a function of the accuracy $\varepsilon$ (up to constants) and hence are sharp in that sense.

These are the first deterministic results for global rates of convergence to approximate first- and second-order Karush--Kuhn--Tucker points on manifolds. They apply in particular for optimization constrained to compact submanifolds of $\Rn$, under simpler assumptions.

\vspace{2mm}
{\scriptsize
Published in IMA Journal of Numerical Analysis, \url{https://doi.org/10.1093/imanum/drx080}.
}

\end{abstract}

\section{Introduction}

Optimization on manifolds is concerned with solving nonlinear and typically nonconvex computational problems of the form
\begin{align}
	\min_{x \in \calM} \ f(x),
	\label{eq:P}
	\tag{P}
\end{align}
where $\calM$ is a (smooth) Riemannian manifold and $f \colon \calM \to \reals$ is a (sufficiently smooth) cost function~\citep{Gabay1982,smith1994optimization,edelman1998geometry,AMS08}. Applications abound in machine learning, computer vision, scientific computing, numerical linear algebra, signal processing, etc. In typical applications, $x$ is a matrix and $\calM$ could be a Stiefel manifold of orthonormal frames (including spheres and groups of rotations), a Grassmann manifold of subspaces,
a cone of positive definite matrices, or simply a Euclidean space such as $\Rn$.

The standard theory for optimization on manifolds takes the standpoint that optimizing on a manifold $\calM$ is not fundamentally different from optimizing in $\Rn$. Indeed, many classical algorithms from unconstrained nonlinear optimization such as gradient descent, nonlinear conjugate gradients, BFGS, Newton's method and trust-region methods~\citep{nocedal1999optimization,ruszczynski2006nonlinear} have been adapted to apply to the larger framework of~\eqref{eq:P}~\citep{adler2002spine,genrtr,AMS08,ring2012optimization,huang2015broyden,sato2014riemcg}.
Software-wise, a few general toolboxes for optimization on manifolds exist now, e.g., Manopt~\citep{manopt}, PyManopt~\citep{pymanopt} and ROPTLIB~\citep{ROPTLIB}.

As~\eqref{eq:P} is typically nonconvex, one does not expect general purpose, efficient algorithms to converge to global optima of~\eqref{eq:P} in general. Indeed, the class of problems~\eqref{eq:P} includes known NP-hard problems. Even computing \emph{local} optima is NP-hard in general~\citep[\S5]{vavasis1991nonlinear}.

Nevertheless, one may still hope to compute points of $\calM$ which satisfy first- and second-order necessary optimality conditions. These conditions take up the same form as in unconstrained nonlinear optimization, with \emph{Riemannian} notions of gradient and Hessian.
For $\calM$ defined by equality constraints, these conditions are equivalent to first- and second-order Karush--Kuhn--Tucker (KKT) conditions, but are simpler to manipulate because the Lagrangian multipliers are automatically determined.

The proposition below states these necessary optimality conditions. Recall that to each point $x$ of $\calM$ corresponds a tangent space (a linearization) $\T_x\calM$. The Riemannian gradient $\grad f(x)$ is the unique tangent vector at $x$ such that $\D f(x)[\eta] = \inner{\eta}{\grad f(x)}$ for all tangent vectors $\eta$, where $\inner{\cdot}{\cdot}$ is the Riemannian metric on $\T_x\calM$, and $\D f(x)[\eta]$ is the directional derivative of $f$ at $x$ along $\eta$. The Riemannian Hessian $\Hess f(x)$ is a symmetric operator on $\T_x\calM$, corresponding to the derivative of the gradient vector field with respect to the Levi--Civita connection---see~\citep[\S5]{AMS08}.
These objects are easily computed in applications.
A summary of relevant concepts about manifolds can be found in Appendix~\ref{apdx:manifolds}.

\begin{proposition}[Necessary optimality conditions]\label{prop:necessaryconditions}
	Let $x\in\calM$ be a local optimum for~\eqref{eq:P}. If~$f$ is differentiable at $x$, then $\grad f(x) = 0$. If $f$ is twice differentiable at $x$, then $\Hess f(x) \succeq 0$ (positive semidefinite).
\end{proposition}
\begin{proof}
	See~\cite[Rem.~4.2 and Cor.~4.2]{yang2012optimality}.
		\end{proof}
A point $x\in\calM$ which satisfies $\grad f(x) = 0$ is a \emph{(first-order) critical point} (also called a stationary point). If $x$ furthermore satisfies $\Hess f(x) \succeq 0$, it is a \emph{second-order critical point}.

Existing theory for optimization algorithms on manifolds is mostly concerned with establishing global convergence to critical points
without rates (where global means regardless of initialization), as well as local rates of convergence. For example, gradient descent is known to converge globally to critical points, and the convergence rate is linear once the iterates reach \emph{a sufficiently small neighborhood} of the limit point~\cite[\S4]{AMS08}.
Early work of \citet{udriste1994convex} on local convergence rates even bounds distance to optimizers as a function of iteration count, assuming initialization in a set where the Hessian of $f$ is positive definite, with lower and upper bounds on the eigenvalues; see also~\cite[Thm.~4.5.6, Thm.~7.4.11]{AMS08}. 
Such guarantees adequately describe the empirical behavior of those methods, but give no information about how many iterations are required to reach the local regime from an arbitrary initial point $x_0$; that is: the worst-case scenarios are not addressed.

For classical unconstrained nonlinear optimization, this caveat has been addressed by bounding the number of iterations required by known algorithms to compute points which satisfy necessary optimality conditions within some tolerance, without assumptions on the initial iterate.
Among others, \citet{nesterov2004introductory} gives a proof that, for $\calM = \Rn$ and Lipschitz differentiable $f$, gradient descent with an appropriate step-size computes a point $x$ where $\|\grad f(x)\| \leq \varepsilon$ in $\mathcal{O}(1/\varepsilon^2)$ iterations. This is sharp~\citep{cartis2010steepestdescent}. \citet{cartis2012complexity} prove the same for trust-region methods, and further show that if $f$ is twice Lipschitz continuously differentiable, then a point $x$ where $\|\grad f(x)\| \leq \varepsilon$ and $\Hess f(x) \succeq -\varepsilon \Id$ is computed in $\mathcal{O}(1/\varepsilon^3)$ iterations, also with examples showing sharpness.

In this paper, we extend the unconstrained results to the larger class of optimization problems on manifolds~\eqref{eq:P}. This work builds upon the original proofs~\citep{nesterov2004introductory,cartis2012complexity} and on existing adaptations of gradient descent and trust-region methods to manifolds~\citep{genrtr,AMS08}. One key step is the identification of a set of relevant Lipschitz-type regularity assumptions which allows the proofs to carry over from $\Rn$ to $\calM$ with relative ease.

\subsection*{Main results}

We state the main results here informally. We use the notion of \emph{retraction} $\Retr_x$ (see Definition~\ref{def:retraction} below), which allows to
map tangent vectors at $x$ to points on $\calM$.
Iterates are related by $x_{k+1} = \Retr_{x_k}(\eta_k)$ for some tangent vector $\eta_k$ at $x_k$ (the step). Hence, $f\circ \Retr_x$ is a lift of the cost function from $\calM$ to the tangent space at $x$. For $\calM = \Rn$, the standard retraction gives $\Retr_{x_k}(\eta_k) = x_k + \eta_k$. By $\|\cdot\|$, we denote the norm associated to the Riemannian metric. 
\paragraph{About gradient descent} (See Theorems~\ref{thm:lipschitzdescent} and~\ref{thm:armijodescent}.)
For problem~\eqref{eq:P}, if $f$ is bounded below on $\calM$ and $f \circ \Retr_x$ has Lipschitz gradient with constant $\L$ independent of $x$, then Riemannian gradient descent with constant step size $1/\L$ or with backtracking Armijo line-search returns $x$ with $\|\grad f(x)\| \leq \varepsilon$ in $\mathcal{O}(1/\varepsilon^2)$ iterations. 	

\paragraph{About trust regions} (See Theorem~\ref{thm:rtr-scnd-main}.)
For problem~\eqref{eq:P}, if $f$ is bounded below on $\calM$ and $f \circ \Retr_x$ has Lipschitz gradient with constant independent of $x$, then RTR returns $x$ with $\|\grad f(x)\| \leq \varepsilon_g$ in $\mathcal{O}(1/\varepsilon_g^2)$ iterations, under weak assumptions on the model quality. If further $f \circ \Retr_x$ has Lipschitz Hessian with constant independent of $x$, then RTR returns $x$ with $\|\grad f(x)\| \leq \varepsilon_g$ and $\Hess f(x) \succeq -\varepsilon_H \Id$ in $\mathcal{O}(\max\{1/\varepsilon_H^3, 1/\varepsilon_g^2 \varepsilon_H^{}\})$ iterations, provided the true Hessian is used in the model and a second-order retraction is used.

\paragraph{About compact submanifolds} (See Lemmas~\ref{lem:compactsub1} and~\ref{lem:compactsub2}.)
The first-order regularity conditions above hold in particular if $\calM$ is a compact submanifold of a Euclidean space $\calE$ (such as $\Rn$) and $f \colon \calE \to \reals$ has a locally Lipschitz continuous gradient. The second-order regularity conditions hold if furthermore $f$ has a locally Lipschitz continuous Hessian on $\calE$ and the retraction is second order (Definition~\ref{def:retraction2}).

\paragraph{}

Since the rates $\mathcal{O}(1/\varepsilon^2)$ and $\mathcal{O}(1/\varepsilon^3)$ are sharp for gradient descent and trust regions when $\calM = \Rn$~\citep{cartis2010steepestdescent,cartis2012complexity}, they are also sharp for $\calM$ a generic Riemannian manifold. Below, constants are given explicitly, thus precisely bounding the total amount of work required in the worst case to attain a prescribed tolerance.

The theorems presented here are the first deterministic results about the worst-case iteration complexity of computing (approximate) first- and second-order critical points on manifolds. The choice of analyzing Riemannian gradient descent and RTR first is guided by practical concerns, as these are among the most commonly used methods on manifolds so far. 
The proposed complexity bounds are particularly relevant when applied to problems for which second-order necessary optimality conditions are also sufficient. See for example~\citep{sun2015complete,sun2016geometric,boumal2015staircase,boumal2016nonconvexphase,bandeira2016lowrankmaxcut,bhojanapalli2016global,ge2016matrix} and the example in Section~\ref{sec:example}.

\subsection*{Related work}

The complexity of Riemannian optimization is discussed in a few recent lines of work. \citet{zhang2016complexitygeodesicallyconvex} treat geodesically convex problems over Hadamard manifolds. This is a remarkable extension of important pieces of classic convex optimization theory to manifolds with negative curvature. Because of the focus on geodesically convex problems, those results do not apply to the more general problem~\eqref{eq:P}, but have the clear advantage of guaranteeing global optimality. In \citep{zhang2016riemannian}, which appeared a day before the present paper on public repositories, the authors also study the iteration complexity of nonconvex optimization on manifolds.
Their results differ from the ones presented here in that they focus on \emph{stochastic} optimization algorithms, aiming for first-order conditions. Their results assume bounded curvature for the manifold. Furthermore, their analysis relies on the Riemannian exponential map, whereas we cover the more general class of retraction maps (which is computationally advantageous). We also do not use the notions of Riemannian parallel transport or logarithmic map, which, in our view, makes for a simpler analysis.

\citet{sun2015complete,sun2016geometric} consider dictionary learning and phase retrieval, and show that these problems, when appropriately framed as optimization on a manifold, are low dimensional and have no spurious local optimizers. They derive the complexity of RTR specialized to their application. In particular, they combine the global rate with a local convergence rate, which allows them to establish an overall better complexity than $\mathcal{O}(1/\varepsilon^3)$, but with an idealized version of the algorithm and restricted to these relevant applications. In this paper, we favor a more general approach, focused on algorithms closer to the ones implemented in practice.

Recent work by~\citet{bento2017iterationcomplexity} (which appeared after a first version of this paper) focuses on iteration complexity of gradient, subgradient and proximal point methods for the case of convex cost functions on manifolds, using the exponential map as retraction.

For the unconstrained case, optimal complexity bounds of order $\calO(1/\varepsilon^{1.5})$ to generate $x$ with $\|\grad f(x)\| \leq \varepsilon$ have also been given for cubic regularization methods~\citep{cartis2011adaptivecubic,cartis2011optimal}
and sophisticated trust region variants~\citep{CurtRobiSama14}. Bounds for regularization methods can be further improved given higher-order derivatives~\citep{bcgmt}.

Worst-case evaluation complexity bounds have been extended to constrained smooth problems in~\citep{cartis2014complexityconstrained,cartis2015evaluation,cartis2015evaluationcomplexity}. There, it is shown that
some carefully devised, albeit impractical, phase~1--phase~2 methods can compute approximate KKT points with global rates of convergence of the same order as in the unconstrained case.
We note that when the constraints are convex (but the objective may not be), practical, feasible methods have been devised~\citep{cartis2015evaluation} that connect to our approach below.
Second-order optimality for the case of convex constraints with nonconvex cost is recently addressed in~\citep{cartis2017secondorder}.

\section{Riemannian gradient descent methods}

Consider the generic Riemannian descent method described in Algorithm~\ref{algo:genericdescentframework}. We first prove that, provided sufficient decrease in the cost function is achieved at each iteration, the algorithm computes a point $x_k$ such that $\|\grad f(x_k)\| \leq \varepsilon$ with $k = \calO(1/\varepsilon^2)$. Then, we propose a Lipschitz-type assumption which is sufficient to guarantee that simple strategies to pick the steps $\eta_k$ indeed ensure sufficient decrease. The proofs parallel the standard ones~\cite[\S1.2.3]{nesterov2004introductory}. The main novelty is the careful extension to the Riemannian setting, which requires the well-known notion of retraction (Definition~\ref{def:retraction}) and the new assumption \aref{assu:lipschitzpullbackrestricted} (see below).

The step $\eta_k$ is a tangent vector to $\calM$ at $x_k$. Because $\calM$ is nonlinear (in general), the operation $x_k+\eta_k$ is undefined. The notion of \emph{retraction} provides a theoretically sound replacement. Informally, $x_{k+1} = \Retr_{x_k}(\eta_k)$ is \emph{a point on $\calM$ one reaches by moving away from $x_k$, along the direction $\eta_k$, while remaining on the manifold}. The Riemannian exponential map (which generates geodesics) is a retraction. The crucial point is that many other maps are retractions, often far less difficult to compute than the exponential. The definition of retraction below can be traced back to~\citet{shub1986some} and it appears under that name in~\citep{adler2002spine}; see also~\citep[Def.~4.1.1 and \S4.10]{AMS08} for additional references.
\begin{definition}[Retraction]\label{def:retraction}
	A \emph{retraction} on a manifold $\calM$ is a smooth mapping $\Retr$ from the tangent bundle\footnote{Informally, the tangent bundle $\T\calM$ is the set of all pairs $(x, \eta_x)$ where $x\in\calM$ and $\eta_x\in\T_x\calM$. See~\citep{AMS08} for a proper definition of $\T\calM$ and of what it means for $\Retr$ to be smooth.} $\T\calM$ to $\calM$ with the following properties. Let $\Retr_x \colon \T_x\calM \to \calM$ denote the restriction of $\Retr$ to $\T_x\calM$.
	\begin{enumerate}
	\item[(i)] $\Retr_x(0_x) = x$, where $0_x$ is the zero vector in $\T_x\calM$;
	\item[(ii)] The differential of $\Retr_x$ at $0_x$, $\D\Retr_x(0_x)$, is the identity map.
	\end{enumerate}
	These combined conditions ensure retraction curves $t \mapsto \Retr_x(t\eta)$ agree up to first order with geodesics passing through $x$ with velocity $\eta$, around $t=0$. Sometimes, we allow $\Retr_x$ to be defined only locally, in a closed ball of radius $\varrho(x) > 0$ centered at $0_x$ in $\T_x\calM$.
\end{definition}
In linear spaces
such as $\Rn$, the typical choice is $\Retr_x(\eta) = x+\eta$. On the sphere, a popular choice is $\Retr_x(\eta) = \frac{x+\eta}{\|x+\eta\|}$. \begin{remark} \label{rem:domainrestriction}
		If the retraction at $x_k$ is only defined in a ball of radius $\varrho_k = \varrho(x_k)$ around the origin in $\T_{x_k}\calM$, we limit the size of step $\eta_k$ to $\varrho_k$. Theorems in this section provide a complexity result provided $\varrho = \inf_k \varrho_k > 0$. If the \emph{injectivity radius} of the manifold is positive, retractions satisfying the condition $\inf_{x\in\calM} \varrho(x) > 0$ exist. In particular, compact manifolds have positive injectivity radius~\citep[Thm.\,III.2.3]{chavel2006riemannian}.
		The option to limit the step sizes is also useful when the constant $\L$ in~\aref{assu:lipschitzpullbackrestricted} below does not exist globally. \end{remark}

\begin{algorithm}[h]
	\caption{Generic Riemannian descent algorithm}
	\label{algo:genericdescentframework}
	\begin{algorithmic}[1]
		\State \textbf{Given:} $f\colon\calM\to\reals$ differentiable, a retraction $\Retr$ on $\calM$, $x_0 \in \calM$, $\varepsilon > 0$ 		\State \textbf{Init:} $k \leftarrow 0$
		\While{$\|\grad f(x_k)\| > \varepsilon$}
		\State Pick $\eta_k \in \T_{x_k}\calM$ (e.g., as in Theorem~\ref{thm:lipschitzdescent} or Theorem~\ref{thm:armijodescent})
		\State $x_{k+1} = \Retr_{x_k}(\eta_k)$
		\State $k \leftarrow k+1$
		\EndWhile
		\State \Return $x_k$ \Comment{$\|\grad f(x_k)\| \leq \varepsilon$}
	\end{algorithmic}
\end{algorithm}

The two central assumptions and a general theorem about Algorithm~\ref{algo:genericdescentframework} follow.
\begin{assumption}[Lower bound]\label{assu:lowerbound} 
	There exists $f^* > -\infty$ such that $f(x) \geq f^*$ for all $x \in \calM$.
\end{assumption}
\begin{assumption}[Sufficient decrease]\label{assu:sufficientdecrease}
	There exist $c, c' > 0$ such that, for all $k \geq 0$, 	\begin{align*}
		f(x_k) - f(x_{k+1}) \geq \min\left( c \|\grad f(x_k)\|, c'  \right) \|\grad f(x_k)\|.
	\end{align*}
\end{assumption}

\begin{theorem}\label{thm:masterdescent}
	Under~\aref{assu:lowerbound} and~\aref{assu:sufficientdecrease}, Algorithm~\ref{algo:genericdescentframework} returns $x\in\calM$ satisfying $f(x) \leq f(x_0)$ and $\|\grad f(x)\| \leq \varepsilon$ in at most
	\begin{align*}
		\left\lceil \frac{f(x_0)-f^*}{c} \cdot \frac{1}{\varepsilon^2} \right\rceil
	\end{align*}
	iterations, provided $\varepsilon \leq \frac{c'}{c}$. If $\varepsilon > \frac{c'}{c}$, at most $\left\lceil \frac{f(x_0)-f^*}{c'} \cdot \frac{1}{\varepsilon} \right\rceil$ iterations are required.
		\end{theorem}
\begin{proof}	
	If Algorithm~\ref{algo:genericdescentframework} executes $K-1$ iterations without terminating, then $\|\grad f(x_k)\| > \varepsilon$ for all $k$ in $0, \ldots, K-1$. Then, using \aref{assu:lowerbound} and \aref{assu:sufficientdecrease} in a classic telescoping sum argument gives:
	\begin{align*}
		f(x_0) - f^* \geq f(x_0) - f(x_K) & = \sum_{k=0}^{K-1} f(x_k) - f(x_{k+1}) > K \min(c\varepsilon, c') \varepsilon.
	\end{align*}
	By contradiction, the algorithm must have terminated if $K \geq \frac{f(x_0)-f^*}{\min(c\varepsilon, c')\varepsilon}$.
	\\\\\end{proof}

To ensure~\aref{assu:sufficientdecrease} with simple rules for the choice of $\eta_k$, it is necessary to restrict the class of functions $f$. For the particular case $\calM = \Rn$ and $\Retr_x(\eta) = x+\eta$, the classical assumption is to require $f$ to have a Lipschitz continuous gradient~\citep{nesterov2004introductory}, that is, existence of $\Lg$ such that: \begin{align}
	\forall x, y \in \Rn, \quad \| \grad f(x) - \grad f(y) \| \leq \Lg \|x - y\|.
	\label{eq:LipschitzGradientRn}
\end{align}
As we argue momentarily, generalizing this property to manifolds is impractical. On the other hand, it is well known that~\eqref{eq:LipschitzGradientRn} implies (see for example~\citep[Lemma\,1.2.3]{nesterov2004introductory}; see also~\citep[App.~A]{berger2017fastmatrixmult} for a converse):
\begin{align}
	\forall x, y \in \Rn, \quad \left| f(y) - \left[ f(x) + \inner{y-x}{\grad f(x)} \right] \right| \leq \frac{\Lg}{2} \|y-x\|^2.
	\label{eq:TaylorLipschitzRn}
\end{align}
It is the latter we adapt to manifolds. Consider the \emph{pullback}\footnote{The composition $f \circ \Retr_x$ is called the pullback because it, quite literally, pulls back the cost function $f$ from the manifold $\calM$ to the linear space $\T_x\calM$.} $\hat f_x = f \circ \Retr_{x} \colon \T_{x}\calM \to \reals$, conveniently defined on a vector space.
It follows from the definition of retraction that $\grad \hat f_x(0_{x}) = \grad f(x)$.\footnote{$\forall \eta \in \T_{x}\calM, \innersmall{\grad \hat f_x(0_{x})}{\eta} = \D \hat f_x(0_{x})[\eta] = \D f({x})[\D \Retr_{{x}}(0_{x})[\eta]] = \D f({x})[\eta] = \inner{\grad f({x})}{\eta}.$} Thinking of $x$ as $x_k$ and of $y$ as $\Retr_{{x_k}}(\eta)$, we require the following.
\begin{assumption}[Restricted Lipschitz-type gradient for pullbacks]\label{assu:lipschitzpullbackrestricted}
	There exists $\Lg \geq 0$ such that, for all $x_k$ among $x_0, x_1\ldots$ generated by a specified algorithm, the composition $\hat f_k = f \circ \Retr_{x_k}$ satisfies
	\begin{align}
	\big| \hat f_k(\eta) - \left[f(x_k) + \inner{\eta}{\grad f(x_k)}\right] \big| \leq \frac{\Lg}{2} \|\eta\|^2
	\label{eq:lipschitzpullbackLg}
	\end{align}
	for all $\eta \in \T_{x_k}\calM$ such that $\|\eta\| \leq \varrho_k$.\footnote{See Remark~\ref{rem:domainrestriction}; $\rho_k = \infty$ is valid if the retraction is globally defined and $f$ is sufficiently nice (e.g., Lemma~\ref{lem:compactsub1}).}
		In words, the pullbacks $\hat f_k$, possibly restricted to certain balls, are uniformly well approximated by their first-order Taylor expansions around the origin.
\end{assumption}

To the best of our knowledge, this specific assumption has not been used to analyze convergence of optimization algorithms on manifolds before. As will become clear, it allows for simple extensions of existing proofs in $\Rn$.

Notice that, if each $\hat f_k$ has a Lipschitz continuous gradient with constant $\L$ independent of $k$,\footnote{This holds in particular in the classical setting $\calM = \Rn$, $\Retr_x(\eta) = x+\eta$ and $\grad f$ is $\L$-Lipschitz.} then~\aref{assu:lipschitzpullbackrestricted} holds; but the reverse is not necessarily true as~\aref{assu:lipschitzpullbackrestricted} gives a special role to the origin. In this sense, the condition on $\hat f_k$ is weaker than Lipschitz continuity of the gradient of $\hat f_k$. On the other hand, we are requiring this condition to hold for all $x_k$ with the same constant $\L$. This is why we call the condition \emph{Lipschitz-type} rather than Lipschitz.

The following lemma states that if~$\calM$ is a compact submanifold of $\Rn$, then a sufficient condition for~\aref{assu:lipschitzpullbackrestricted} to hold
is for $f\colon\Rn\to\reals$ to have locally Lipschitz continuous gradient (so that it has Lipschitz continuous gradient on any compact subset of $\Rn$). The proof is in Appendix~\ref{apdx:submanifold}.
\begin{lemma}\label{lem:compactsub1}
	Let $\calE$ be a Euclidean space (for example, $\calE = \Rn$) and let $\calM$ be a compact Riemannian submanifold of $\calE$. Let $\Retr$ be a retraction on $\calM$ (globally\footnote{This is typically not an issue in practice. For example, globally defined, practical retractions are known for the sphere, Stiefel manifold, orthogonal group, their products and many others~\citep[\S4]{AMS08}. 
		} defined). If $f \colon \calE \to \reals$ has Lipschitz continuous gradient in the convex hull of $\calM$,
		then the pullbacks $f \circ \Retr_x$ satisfy~\eqref{eq:lipschitzpullbackLg} globally with some constant $\L$ independent of $x$; hence, \aref{assu:lipschitzpullbackrestricted} holds for any sequence of iterates and with $\varrho_k = \infty$ for all $k$.
		\end{lemma}

There are mainly two difficulties with generalizing~\eqref{eq:LipschitzGradientRn} directly to manifolds. Firstly, $\grad f(x)$ and $\grad f(y)$ live in two different tangent spaces, so that their difference is not defined; instead, $\grad f(x)$ must be \emph{transported} to $\T_y\calM$, which requires the introduction of a \emph{parallel transport} $\mathrm{P}_{x \rightarrow y} \colon \T_x\calM \to \T_y\calM$ along a minimal geodesic connecting $x$ and $y$. Secondly, the right hand side $\|x - y\|$ should become $\dist(x, y)$: the \emph{geodesic distance} on $\calM$. Both notions involve subtle definitions and transports may not be defined on all of $\calM$. Overall, the resulting
assumption would read as: there exists $\Lg$ such that
\begin{align}
	\forall x, y \in \calM, \quad \| \mathrm{P}_{x\to y} \grad f(x) - \grad f(y) \| \leq \L \dist(x, y).
	\label{eq:lipschitzimpractical}
\end{align}
It is of course possible to work with~\eqref{eq:lipschitzimpractical}---see for example~\citep[Def.~7.4.3]{AMS08} and recent work of~\citet{zhang2016complexitygeodesicallyconvex,zhang2016riemannian}---but we argue that it is conceptually and computationally advantageous to avoid it if possible. The computational advantage comes from the freedom in~\aref{assu:lipschitzpullbackrestricted} to work with any retraction, whereas parallel transport and geodesic distance are tied to the exponential map.

We note that, if the retraction is the exponential map, then it is known that~\aref{assu:lipschitzpullbackrestricted} holds if~\eqref{eq:lipschitzimpractical} holds---see for example~\citep[Def.~2.2 and Lemma~2.1]{bento2017iterationcomplexity}.

\subsection{Fixed step-size gradient descent method}

Leveraging the regularity assumption~\aref{assu:lipschitzpullbackrestricted}, an easy strategy is to pick the step $\eta_k$ as a fixed scaling of the negative gradient, possibly restricted to a ball of radius $\varrho_k$.
\begin{theorem}[Riemannian gradient descent with fixed step-size]\label{thm:lipschitzdescent}
	Under \aref{assu:lowerbound} and \aref{assu:lipschitzpullbackrestricted}, 	Algorithm~\ref{algo:genericdescentframework} with the explicit strategy
	\begin{align*}
		\eta_k = -\min\left( \frac{1}{\L}, \frac{\varrho_k}{\|\grad f(x_k)\|} \right) \grad f(x_k)
	\end{align*}
	returns a point $x\in\calM$ satisfying $f(x) \leq f(x_0)$ and $\|\grad f(x)\| \leq \varepsilon$ in at most
	\begin{align*}
		\left\lceil 2\big(f(x_0) - f^*\big)\L \cdot  \frac{1}{\varepsilon^2} \right\rceil
	\end{align*}
	iterations provided $\varepsilon \leq \varrho \L$, where $\varrho = \inf_k \rho_k$. If $\varepsilon > \varrho \L$, the algorithm succeeds in at most $\left\lceil 2\big(f(x_0) - f^*\big)  \frac{1}{\varrho} \cdot \frac{1}{\varepsilon} \right\rceil$ iterations.
	Each iteration requires one cost and gradient evaluation, and one retraction.
\end{theorem}
\begin{proof}
	The regularity assumption~\aref{assu:lipschitzpullbackrestricted} provides an upper bound for the pullback for all $k$:
	\begin{align}
		\forall \eta \in \T_{x_k}\calM \textrm{ with } \|\eta\| \leq \varrho_k, \quad f(\Retr_{x_k}(\eta)) \leq f(x_k) + \inner{\eta}{\grad f(x_k)} + \frac{\L}{2} \|\eta\|^2.
		\label{eq:upperboundpullback}
	\end{align}
	For the given choice of $\eta_k$ and using $x_{k+1} = \Retr_{x_k}(\eta_k)$, it follows easily that
	\begin{align*}
		f(x_{k}) - f(x_{k+1}) \\		\geq \min\left( \frac{\|\grad f(x_k)\|}{\L}, \varrho_k \right) \left[ 1 - \frac{\L}{2} \min\left( \frac{1}{\L}, \frac{\varrho_k}{\|\grad f(x_k)\|} \right) \right] \|\grad f(x_k)\|.
	\end{align*}
	The term in brackets is at least $1/2$.
	Thus, \aref{assu:sufficientdecrease} holds with $c = \frac{1}{2\L}$ and $c' = \frac{\varrho}{2}$, allowing to conclude with Theorem~\ref{thm:masterdescent}.
				 \end{proof}
\begin{corollary}
	If there are no step-size restrictions in Theorem~\ref{thm:lipschitzdescent} ($\rho_k \equiv \infty$), the explicit strategy
	$$
		\eta_k = -\frac{1}{\L} \grad f(x_k)
	$$
	returns a point $x\in\calM$ satisfying $f(x) \leq f(x_0)$ and $\|\grad f(x)\| \leq \varepsilon$ in at most
	$$
		\left\lceil 2\big(f(x_0) - f^*\big)\L \cdot  \frac{1}{\varepsilon^2} \right\rceil
	$$
	iterations for any $\varepsilon > 0$.
\end{corollary}

\subsection{Gradient descent with backtracking Armijo line-search}

The following lemma shows that a basic Armijo-type backtracking line-search, Algorithm~\ref{algo:armijo}, computes a step $\eta_k$ satisfying \aref{assu:sufficientdecrease} in a bounded number of function calls, without the need to know $\L$. The statement allows search directions other than $-\grad f(x_k)$, provided they remain ``related'' to $-\grad f(x_k)$. This result is well known in the Euclidean case and carries over seamlessly under \aref{assu:lipschitzpullbackrestricted}.

\begin{algorithm}[h]
	\caption{Backtracking Armijo line-search}  	\label{algo:armijo}
	\begin{algorithmic}[1]
		\State \textbf{Given:} $x_k \in \calM$, $\eta_k^0 \in \T_{x_k}\calM$, ${\bar t}_k > 0$, $c_1 \in (0, 1)$, $\tau \in (0, 1)$
		\State \textbf{Init:} $t \leftarrow {\bar t}_k$
		\While{$f(x_k) - f(\Retr_{x_k}(t \cdot \eta_k^0)) < c_1 t \inner{-\grad f(x_k)}{\eta_k^0}$}
		\State $t \leftarrow \tau \cdot t$
		\EndWhile
		\State \Return $t$ and $\eta_k = t \eta_k^0.$
	\end{algorithmic}
\end{algorithm}

\begin{lemma}\label{lem:armijo}
	For each iteration $k$ of Algorithm~\ref{algo:genericdescentframework}, let $\eta_k^0 \in \T_{x_k}\calM$ be the initial search direction to be considered for line-search. Assume there exist constants $c_2 \in (0, 1]$ and $0 < c_3 \leq c_4$ such that, for all $k$,
	\begin{align*}
		\inner{-\grad f(x_k)}{\eta_k^0} & \geq c_2 \|\grad f(x_k)\|\|\eta_k^0\| & \textrm{ and } & &
		c_3 \|\grad f(x_k)\| & \leq \|\eta_k^0\| \leq c_4 \|\grad f(x_k)\|.
	\end{align*}
	Under \aref{assu:lipschitzpullbackrestricted},
	backtracking Armijo (Algorithm~\ref{algo:armijo}) with initial stepsize ${\bar t}_k$ such that ${\bar t}_k \|\eta_k^0\| \leq \varrho_k$ returns a positive $t$ and $\eta_k = t\eta_k^0$ such that
	\begin{align}
		f(x_k) - f(\Retr_{x_k}(\eta_k)) & \geq c_1c_2c_3t\|\grad f(x_k)\|^2 & \textrm{ and } & & t & \geq \min\left({\bar t}_k, \frac{2\tau c_2 (1-c_1)}{c_4\L}\right)
	\label{eq:suffdecrease}
	\end{align}
	in
	\begin{align*}
		1+\log_\tau\left(t/{\bar t}_k\right) \leq \max\left( 1, 2 + \left\lceil \log_{\tau^{-1}}\left(\frac{c_4{\bar t}_k\L}{2c_2(1-c_1)}\right) \right\rceil \right)
			\end{align*}
	retractions and cost evaluations (not counting evaluation of $f$ at $x_k$).
\end{lemma}
\begin{proof}
	See Appendix~\ref{apdx:lem:armijo}
\end{proof}
The previous discussion can be particularized to bound the amount of work required by a gradient descent method using a backtracking Armijo line-search on manifolds. The constant $\L$ appears in the bounds but needs not be known. Note that, at iteration $k$, the last cost evaluation of the line-search algorithm is the cost at $x_{k+1}$: it needs not be recomputed.
\begin{theorem}[Riemannian gradient descent with backtracking line-search]\label{thm:armijodescent}
	Under~\aref{assu:lowerbound} and~\aref{assu:lipschitzpullbackrestricted}, Algorithm~\ref{algo:genericdescentframework} with Algorithm~\ref{algo:armijo} for line-search using initial search direction $\eta_k^0 = -\grad f(x_k)$ with parameters $c_1, \tau$ and  ${\bar t}_k \triangleq \min\left( {\bar t}, \varrho_k / \|\grad f(x_k)\| \right)$ for some $\bar t > 0$
	returns a point $x\in\calM$ satisfying $f(x) \leq f(x_0)$ and $\|\grad f(x)\| \leq \varepsilon$ in at most
	\begin{align*}
		\left\lceil \frac{f(x_0) - f^*}{c_1 \min\left(\bar t, \frac{2\tau(1-c_1)}{\L}\right)} \cdot  \frac{1}{\varepsilon^2} \right\rceil
	\end{align*}
	iterations, provided $\varepsilon \leq \frac{\varrho}{\min\left(\bar t, \frac{2\tau (1-c_1)}{\L} \right)} \triangleq c$, where $\varrho = \inf_k \varrho_k$. If $\varepsilon > c$, the algorithm succeeds in at most $\left\lceil \frac{f(x_0) - f^*}{c_1 \varrho} \cdot  \frac{1}{\varepsilon} \right\rceil$ iterations.
	After computing $f(x_0)$ and $\grad f(x_0)$, each iteration requires one gradient evaluation and at most $\max\left( 1, 2 + \left\lceil \log_{\tau^{-1}}\left(\frac{{\bar t}\L}{2(1-c_1)}\right) \right\rceil \right)$ cost evaluations and retractions.
\end{theorem}
\begin{proof}
	Using $\eta_k^0 = -\grad f(x_k)$, one can take $c_2 = c_3 = c_4 = 1$ in Lemma~\ref{lem:armijo}. Eq.~\eqref{eq:suffdecrease} in that lemma combined with  the definition of ${\bar t}_k$ ensures
	\begin{align*}
		f(x_k) - f(x_{k+1}) \geq c_1 \min\left( {\bar t}, \frac{2\tau(1-c_1)}{\L}, \frac{\varrho_k}{\|\grad f(x_k)\|} \right) \| \grad f(x_k) \|^2.
	\end{align*}
	Thus,~\aref{assu:sufficientdecrease} holds with $c = c_1 \min\left( {\bar t}, \frac{2\tau(1-c_1)}{\L} \right)$ and $c' = c_1 \varrho$. Conclude with Theorem~\ref{thm:masterdescent}.
\end{proof}

\section{Riemannian trust-region methods}

The Riemannian trust-region method (RTR) is a generalization of the classical trust-region method to manifolds~\citep{genrtr,conn2000trust}---see Algorithm~\ref{algo:RTRscnd}.
The algorithm is initialized with a point $x_0\in\calM$ and a trust-region radius $\Delta_0$. At iteration $k$, the pullback $\hat f_k = f \circ \Retr_{x_k}$ is approximated by a model $\hat m_k \colon \T_{x_k}\calM \to \reals$,
\begin{align}
	\hat m_k(\eta) = f(x_k) + \inner{\eta}{\grad f(x_k)} + \frac{1}{2} \inner{\eta}{H_k[\eta]},
	\label{eq:modelhatm}
\end{align}
where $H_k \colon \T_{x_k}\calM \to \T_{x_k}\calM$ is a map chosen by the user.
The tentative step $\eta_k$ is obtained by approximately solving the associated trust-region subproblem:
\begin{align}
	\min_{\eta \in \T_{x_k}\calM} \ \hat m_k(\eta) \quad \textrm{ subject to } \quad  \|\eta\| \leq \Delta_k.
	\label{eq:TRproblem}
\end{align}
The candidate next iterate $x_k^+ = \Retr_{x_k}(\eta_k)$ is accepted ($x_{k+1} = x_k^+$) if the actual cost decrease $f(x_k) - f(x_k^+)$ is a sufficiently large fraction of the model decrease $\hat m_k(0_{x_k}) - \hat m_k(\eta_k)$. Otherwise, the candidate is rejected ($x_{k+1} = x_k$). Depending on the level of agreement between the model decrease and actual decrease, the trust-region radius $\Delta_k$ can be reduced, kept unchanged or increased, but never above some parameter $\bar\Delta$. The parameter $\bar\Delta$ can be used in particular in case of a non-globally defined retraction or if the regularity conditions on the pullbacks hold only locally.

We establish worst-case iteration complexity bounds for the computation of points $x\in\calM$ such that $\|\grad f(x)\| \leq \varepsilon_g$ and $\Hess f(x) \succeq -\varepsilon_H \Id$, where $\Hess f(x)$ is the Riemannian Hessian of $f$ at $x$. Besides Lipschitz-type conditions on the problem itself, essential algorithmic requirements are that (i) the models $\hat m_k$ should agree sufficiently with the pullbacks $\hat f_k$ (locally); and (ii) sufficient decrease in the model should be achieved at each iteration. The analysis presented here is a generalization of the one in~\citep{cartis2012complexity} to manifolds.

\begin{algorithm}[t]
	\caption{Riemannian trust regions (RTR), modified to attain second-order optimality}
	\label{algo:RTRscnd}
	\begin{algorithmic}[1]
		\State \textbf{Parameters: } $\bar{\Delta} > 0$, $0 < \rho' < 1/4$, $\varepsilon_g > 0$, $\varepsilon_H > 0$
		\State \textbf{Input:} $x_0 \in \calM$, $0 < \Delta_0 \leq \bar{\Delta}$
		\State \textbf{Init:} $k \leftarrow 0$
		\While{\textbf{true}}
		\Statex 
		\If{$\|\grad f(x_k)\| > \varepsilon_g$} \Comment{First-order step.}
		\State Obtain $\eta_k \in \T_{x_k}\calM$ satisfying~\aref{assu:rtr-Cauchy-step} (e.g., Lemma~\ref{lem:CauchyStep})
		\ElsIf{$\varepsilon_H < \infty$} \Comment{Second-order step.}
		\If{$\lambdamin(H_k) < -\varepsilon_H$} \label{step:lambdamincheck}
		\State Obtain $\eta_k \in \T_{x_k}\calM$ satisfying~\aref{assu:rtr-eigenstep} (e.g., Lemma~\ref{lem:EigenStep})
		\Else
		\State \Return $x_k$ \Comment{$\|\grad f(x_k)\| \leq \varepsilon_g$ and $\lambdamin(H_k) \geq -\varepsilon_H$.}
		\EndIf
		\Else
		\State \Return $x_k$ \Comment{$\|\grad f(x_k)\| \leq \varepsilon_g$.}
		\EndIf
		\Statex 
		\State Compute
		\begin{align}
			\rho_k & = \frac{\hat f_k(0_{x_k}) - \hat f_k(\eta_k)}{\hat m_k(0_{x_k}) - \hat m_k(\eta_k)}
			\label{eq:rhok}
		\end{align}
		\Statex
		\State $\Delta_{k+1} = \begin{cases}
		\frac{1}{4} \Delta_k & \textrm{ if } \rho_k < \frac{1}{4} \textrm{ (poor model-cost agreement),} \\
		\min\left( 2\Delta_k, \bar{\Delta} \right) & \textrm{ if } \rho_k > \frac{3}{4} \textrm{ and } \|\eta_k\| = \Delta_k \textrm{ (good agreement, limiting TR),} \\
		\Delta_k & \textrm{ otherwise.}
		\end{cases}$
		\Statex 
		\State $x_{k+1} = \begin{cases}
		\Retr_{x_k}(\eta_k) & \textrm{ if } \rho_k > \rho' \textrm{ (accept the step),} \\
		x_k & \textrm{ otherwise (reject).}
		\end{cases}$
		\Statex 
		\State $k \leftarrow k+1$
		\EndWhile
	\end{algorithmic}
\end{algorithm}

\subsection{Regularity assumptions}

In what follows, for iteration $k$, we make assumptions involving the ball of radius $\Delta_k \leq \bar{\Delta}$ around $0_{x_k}$ in the tangent space at $x_k$.
If $\Retr_x$ is only defined in a ball of radius $\varrho(x)$, one (conservative) strategy to ensure $\varrho_k \geq \Delta_k$ as required in the assumption below is to set $\bar \Delta \leq \inf_{x\in\calM : f(x) \leq f(x_0)} \varrho(x)$, provided this is positive (see Remark~\ref{rem:domainrestriction}).
\begin{assumption}[Restricted Lipschitz-type gradient for pullbacks]\label{assu:rtr-regularity-first}
	Assumption~\aref{assu:lipschitzpullbackrestricted} holds in the respective trust regions of the iterates produced by Algorithm~\ref{algo:RTRscnd}, that is, with $\varrho_k \geq \Delta_k$.
\end{assumption}

\begin{assumption}[Restricted Lipschitz-type Hessian for pullbacks] \label{assu:rtr-regularity-second}
	If $\varepsilon_H < \infty$, there exists $\LH \geq 0$ such that, for all $x_k$ among $x_0, x_1\ldots$ generated by Algorithm~\ref{algo:RTRscnd} and such that $\|\grad f(x_k)\| \leq \varepsilon_g$, $\hat f_k$ satisfies
	\begin{align}
		\left| \hat f_k(\eta) - \left[ f(x_k) + \inner{\eta}{\grad f(x_k)} + \frac{1}{2}\innersmall{\eta}{\nabla^2 \hat f_k(0_{x_k})[\eta]} \right] \right| \leq \frac{\LH}{6} \|\eta\|^3
		\label{eq:secondorderreg}
	\end{align}
	for all \mbox{$\eta \in \T_{x_k}\calM$} such that $\|\eta\| \leq \Delta_k$.
\end{assumption}
As discussed in Section~\ref{sec:HandHess} below, if $\Retr$ is a \emph{second-order} retraction, then $\nabla^2 \hat f_k(0_{x_k})$ coincides with the Riemannian Hessian of $f$ at $x_k$.

In the previous section, Lemma~\ref{lem:compactsub1} gives a sufficient condition for~\aref{assu:rtr-regularity-first} to hold; we complement this statement with a sufficient condition for~\aref{assu:rtr-regularity-second} to hold as well. In a nutshell: if $\calM$ is a compact submanifold of $\Rn$ and $f\colon\Rn\to\reals$ has locally Lipschitz continuous Hessian, then both assumptions hold.
\begin{lemma}\label{lem:compactsub2}
	Let $\calE$ be a Euclidean space (for example, $\calE = \Rn$) and let $\calM$ be a compact Riemannian submanifold of $\calE$. Let $\Retr$ be a second-order retraction on $\calM$ (globally defined). If \mbox{$f \colon \calE \to \reals$} has Lipschitz continuous Hessian in the convex hull of $\calM$,
		then the pullbacks \mbox{$f \circ \Retr_x$} obey~\eqref{eq:secondorderreg} with some constant $\LH$ independent of $x$; hence, \aref{assu:rtr-regularity-second} holds for any sequence of iterates and trust-region radii.
\end{lemma}
The proof is in Appendix~\ref{apdx:submanifold}. Here too, if $\calM$ is a Euclidean space and $\Retr_x(\eta) = x+\eta$, then~\aref{assu:rtr-regularity-first} and~\aref{assu:rtr-regularity-second} are satisfied if $f$ has Lipschitz continuous Hessian in the usual sense.

\subsection{Assumptions about the models}

The model at iteration $k$ is the function $\hat m_k$~\eqref{eq:modelhatm} whose purpose is to approximate the pullback $\hat f_k = f \circ \Retr_{x_k}$. It involves a map $H_k \colon \T_{x_k}\calM \to \T_{x_k}\calM$. Depending on the type of step being performed (aiming for first- or second-order optimality conditions), we require different properties of the maps $H_k$. Conditions for first-order optimality are particularly lax.

\begin{assumption}\label{assu:rtr-Hk-first}
	If $\|\grad f(x_k)\| > \varepsilon_g$ (so that we are only aiming for a first-order condition at this step), then $H_k$ is radially linear. That is,
	\begin{align}
		\forall \eta\in\T_{x_k}\calM, \forall \alpha \geq 0, \quad H_k[\alpha \eta] = \alpha H_k[\eta].
	\end{align}
	Furthermore, there exists $c_0 \geq 0$ (the same for all first-order steps) such that
	\begin{align}
		\|H_k\| \triangleq \sup_{\eta \in \T_{x_k}\calM : \|\eta\| \leq 1} \inner{\eta}{H_k[\eta]} \leq c_0.
		\label{eq:Hkbound}
	\end{align}
	\end{assumption}
Radial linearity and boundedness are sufficient to ensure first-order agreement between $\hat m_k$ and $\hat f_k$. This relaxation from complete linearity of $H_k$---which would be the standard assumption---notably allows the use of nonlinear finite difference approximations of the Hessian~\citep{boumal2015rtrfd}. To reach second-order agreement, the conditions are stronger.
\begin{assumption}\label{assu:rtr-Hk-second}
	If $\|\grad f(x_k)\| \leq \varepsilon_g$ and $\varepsilon_H < \infty$ (so that we are aiming for a second-order condition), then $H_k$ is \emph{linear and symmetric}. Furthermore, $H_k$ is close to $\nabla^2 \hat f_k(0_{x_k})$
	along $\eta_k$
	in the sense that there exists $c_1 \geq 0$ (the same for all second-order steps) such that:
	\begin{align}
				\left| \inner{\eta_k}{\big(\nabla^2 \hat f_k(0_{x_k}) - H_k\big)[\eta_k]} \right| \leq \frac{c_1 \Delta_k}{3} \|\eta_k\|^2.
		\label{eq:assurtrHksecond}
	\end{align}
	\end{assumption}
The smaller $\Delta_k$, the more precisely $H_k$ must approximate the Hessian of the pullback along $\eta_k$.
	Lemma~\ref{lemma:rtrscnd-delta-lowerbound} (below) shows $\Delta_k$ is lower-bounded in relation with $\varepsilon_g$ and $\varepsilon_H$.

Eq.~\eqref{eq:assurtrHksecond} involves $\eta_k$, the ultimately chosen step which typically depends on $H_k$. The stronger condition below does not reference $\eta_k$ yet still ensures~\eqref{eq:assurtrHksecond} is satisfied:
\begin{align*}
		\left\|\nabla^2 \hat f_k(0_{x_k}) - H_k\right\| \leq \frac{c_1 \Delta_k}{3}.
\end{align*}
Refer to Section~\ref{sec:HandHess} to relate $H_k$, $\nabla^2 \hat f_k(0_{x_k})$ and $\Hess f(x_k)$.

\subsection{Assumptions about sufficient model decrease}

The steps $\eta_k$ can be obtained in a number of ways, leading to different local convergence rates and empirical performance.
As far as global convergence guarantees are concerned though, the requirements are modest. It is only required that, at each iteration, the candidate~$\eta_k$ induces sufficient decrease \emph{in the model}. Known explicit strategies achieve these decreases.
In particular, solving the trust-region subproblem~\eqref{eq:TRproblem} within some tolerance (which can be done in polynomial time if $H_k$ is linear~\citep[\S4.3]{vavasis1991nonlinear}) is certain to satisfy the assumptions. The Steihaug--Toint truncated conjugate gradients method is a popular choice~\citep{toint81towards,steihaug1983conjugate,conn2000trust,genrtr}. See also~\citep{sorensen1982newton,more1983trustregionstep} for more about the trust-region subproblem.
Here, we describe simpler yet satisfactory strategies.
For first-order steps, we require the following.
\begin{assumption}\label{assu:rtr-Cauchy-step}
	There exists $c_2 > 0$ such that, for all $k$ such that $\|\grad f(x_k)\| > \varepsilon_g$,
	the step $\eta_k$ satisfies
	\begin{align}
		\hat m_k(0_{x_k}) - \hat m_k(\eta_k) & \geq c_2 \min\left( \Delta_k, \frac{\varepsilon_g}{c_0} \right) \varepsilon_g.
	\end{align}
\end{assumption}
As is well known, the explicitly computable \emph{Cauchy step} satisfies this requirement. For convenience, let $g_k = \grad f(x_k)$. By definition, the Cauchy step minimizes $\hat m_k$~\eqref{eq:modelhatm} in the trust region along the steepest descent direction $-g_k$. Owing to radial linearity~(\aref{assu:rtr-Hk-first}), this reads:
\begin{align*}
	\min_{\alpha \geq 0} & \ \hat m_k(-\alpha g_k) = f(x_k) - \alpha \|g_k\|^2 + \frac{\alpha^2}{2} \inner{g_k}{H_k[g_k]} \\
	\textrm{ s.t.} & \ \alpha \|g_k\| \leq \Delta_k.\nonumber
\end{align*}
This corresponds to minimizing a quadratic in $\alpha$ over the interval $[0, \Delta_k / \|g_k\|]$. The optimal value is easily seen to be~\citep{conn2000trust}
\begin{align*}
	\alpha_k^C & =
	\begin{cases}
		\min\left( \frac{\|g_k\|^2}{\inner{g_k}{H_k[g_k]}}, \frac{\Delta_k}{\|g_k\|} \right) & \textrm{ if } \inner{g_k}{H_k[g_k]} > 0, \\
		\frac{\Delta_k}{\|g_k\|} & \textrm{ otherwise.}
	\end{cases}
\end{align*}
\begin{lemma}\label{lem:CauchyStep}
	Let $g_k = \grad f(x_k)$. Under~\aref{assu:rtr-Hk-first},
	setting $\eta_k$ to be the Cauchy step $\eta_k^C = -\alpha_k^C g_k$ for first-order steps fulfills~\aref{assu:rtr-Cauchy-step} with $c_2 = 1/2$. Computing $\eta_k^C$ involves one gradient evaluation and one application of $H_k$.
\end{lemma}
\begin{proof}
	The claim follows as an exercise from $\hat m_k(0_{x_k}) - \hat m_k(\eta_k^C) = \alpha_k^C \|g_k\|^2 - \frac{(\alpha_k^C)^2}{2}\inner{g_k}{H_k[g_k]}$ and  $\inner{g_k}{H_k[g_k]} \leq c_0 \|g_k\|^2$ owing to~\aref{assu:rtr-Hk-first}.
\end{proof}
The Steihaug--Toint truncated conjugate gradient method~\citep{toint81towards,steihaug1983conjugate} is a monotonically improving iterative method for the trust-region subproblem whose first iterate is the Cauchy step; as such, it necessarily achieves the required model decrease.

For second-order steps, the requirement is as follows.
\begin{assumption}\label{assu:rtr-eigenstep}
	There exists $c_3 > 0$ such that, 		for all $k$ such that $\|\grad f(x_k)\| \leq \varepsilon_g$ and $\lambdamin(H_k) < -\varepsilon_H$,
	the step $\eta_k$ satisfies
	\begin{align}
		\hat m_k(0_{x_k}) - \hat m_k(\eta_k) & \geq c_3 \Delta_k^2 \varepsilon_H.
	\end{align}
\end{assumption}
This can be achieved
by making a step of maximal length along a direction which certifies that $\lambdamin(H_k) < -\varepsilon_H$~\citep{conn2000trust}: this is called an \emph{eigenstep}.
Like Cauchy steps, eigensteps can be computed in a finite number of operations, independently of $\varepsilon_g$ and $\varepsilon_H$.
\begin{lemma}\label{lem:EigenStep}
	Under~\aref{assu:rtr-Hk-second}, if $\lambdamin(H_k) < -\varepsilon_H$, there exists a tangent vector $u_k\in\T_{x_k}\calM$ with
	\begin{align*}
		\|u_k\| & = 1, & \inner{u_k}{\grad f(x_k)} & \leq 0, & & \textrm{ and } & \inner{u_k}{H_k[u_k]} < -\varepsilon_H.
	\end{align*}
	Setting $\eta_k$ to be any eigenstep $\eta_k^E = \Delta_k u_k$ for second-order steps fulfills~\aref{assu:rtr-eigenstep} with $c_3 = 1/2$.
	
	Let $v_1, \ldots, v_n$ be an orthonormal basis of $\T_{x_k}\calM$, where $n = \dim\calM$. One way of computing $\eta_k^E$ involves the application of $H_k$ to $v_1, \ldots, v_n$ plus $\mathcal{O}(n^3)$ arithmetic operations. The amount of work is independent of $\varepsilon_g$ and $\varepsilon_H$.
\end{lemma}
\begin{proof}
		Compute $H$, a symmetric matrix of size $n$ which represents $H_k$ in the basis $v_1, \ldots, v_n$, as $H_{ij} = \inner{v_i}{H_k[v_j]}$.
	Compute a factorization $LDL\transpose = H + \varepsilon_H \In$ where $I$ is the identity matrix, $L$ is invertible and triangular, and $D$ is block diagonal with blocks of size $1\times 1$ and $2\times 2$. The factorization can be computed in $\calO(n^3)$ operations~\citep[\S4.4]{golub2012matrix}---see the reference for a word of caution regarding pivoting for stability; pivoting is easily incorporated in the present argument. $D$ has the same inertia as $H+\varepsilon_H \In$, hence $D$ is not positive semidefinite (otherwise $H \succeq -\varepsilon_H \In$.) The structure of $D$ makes it easy to find $x \in \Rn$ with $x\transpose D x < 0$. Solve the triangular system $L\transpose y = x$ for $y\in\Rn$. Now, $0 > x\transpose D x = y\transpose LDL\transpose y = y\transpose (H + \varepsilon_H \In) y$. Consequently, $y\transpose H y < -\varepsilon_H \|y\|^2$. We can set $u_k = \pm \sum_{i=1}^n y_iv_i / \|y\|$, where the sign is chosen to ensure $\inner{u_k}{\grad f(x_k)} \leq 0$. To conclude, check that $\hat m_k(0_{x_k}) - \hat m_k(\eta_k^E) = -\inner{\eta_k^E}{\grad f(x_k)} - \frac{1}{2}\inner{\eta_k^E}{H_k[\eta_k^E]} \geq \frac{1}{2} \Delta_k^2 \varepsilon_H$.
\end{proof}
Notice from the proof that this strategy either certifies that $\lambdamin(H_k) \succeq -\varepsilon_H \Id$ (which must be checked at step~\ref{step:lambdamincheck} in Algorithm~\ref{algo:RTRscnd}) or certifies otherwise by providing an escape direction. We further note that, in practice, one usually prefers to use iterative methods to compute an approximate leftmost eigenvector of $H_k$ without representing it as a matrix.

\subsection{Main results and proofs for RTR}

Under the discussed assumptions, we now establish our main theorem about computation of approximate first- and second-order critical points for~\eqref{eq:P} using RTR in a bounded number of iterations. The following constants will be useful:
\begin{align}
	\lambda_g & = \frac{1}{4}\min\left( \frac{1}{c_0}, \frac{c_2}{\Lg + c_0} \right) & & \textrm{ and } & \lambda_H & = \frac{3}{4} \frac{c_3}{\LH + c_1}.
	\label{eq:rtr-lambda-gH}
\end{align}
\begin{theorem}\label{thm:rtr-scnd-main}
	Under \aref{assu:lowerbound}, \aref{assu:rtr-regularity-first}, \aref{assu:rtr-Hk-first}, \aref{assu:rtr-Cauchy-step} and assuming $\varepsilon_g \leq \frac{\Delta_0}{\lambda_g}$,\footnote{Theorem~\ref{thm:rtr-scnd-main} is scale invariant, in that if the cost function $f(x)$ is replaced by $\alpha f(x)$ for some positive $\alpha$ (which does not meaningfully change~\eqref{eq:P}), it is sensible to also multiply $\Lg, \LH, c_0, c_1, \varepsilon_g$ and $\varepsilon_H$ by $\alpha$; consequently, the upper bounds on $\varepsilon_g$ and $\varepsilon_H$ and the upper bounds on $N_1$ and $N_2$ are invariant under this scaling. If it is desirable to always allow $\varepsilon_g, \varepsilon_H$ in, say, $(0, 1]$, one possibility is to artificially make $\Lg, \LH, c_0, c_1$ larger (which is always allowed).} Algorithm~\ref{algo:RTRscnd} produces an iterate $x_{N_1}$ satisfying $\|\grad f(x_{N_1})\| \leq \varepsilon_g$ with
	\begin{align}
		N_1 \leq  \frac{3}{2} \frac{f(x_0) - f^*}{\rho' c_2 \lambda_g} \frac{1}{\varepsilon_g^2} + \frac{1}{2} \log_2\left( \frac{\Delta_0}{\lambda_g \varepsilon_g} \right)  = \mathcal{O} \left( \frac{1}{\varepsilon_g^2} \right).
	\end{align}
	Furthermore, if $\varepsilon_H < \infty$, then under additional assumptions \aref{assu:rtr-regularity-second}, \aref{assu:rtr-Hk-second}, \aref{assu:rtr-eigenstep} and assuming $\varepsilon_g \leq \frac{c_2}{c_3} \frac{\lambda_H}{\lambda_g^2}$ and $\varepsilon_H \leq \frac{c_2}{c_3} \frac{1}{\lambda_g}$, Algorithm~\ref{algo:RTRscnd} also produces an iterate $x_{N_2}$ satisfying $\|\grad f(x_{N_2})\| \leq \varepsilon_g$ and $\lambdamin(H_{N_2}) \geq -\varepsilon_H$ with
	\begin{align}
		N_1 \leq N_2 \leq \frac{3}{2} \frac{f(x_0) - f^*}{\rho' c_3 \lambda^2} \frac{1}{\varepsilon^2 \varepsilon_H} + \frac{1}{2}\log_2\left( \frac{\Delta_0}{\lambda \varepsilon} \right) = \mathcal{O} \left( \frac{1}{\varepsilon^2 \varepsilon_H} \right),
	\end{align}
	where we defined $(\lambda, \varepsilon) = (\lambda_g, \varepsilon_g)$ if $\lambda_g \varepsilon_g \leq \lambda_H \varepsilon_H$, and $(\lambda, \varepsilon) = (\lambda_H, \varepsilon_H)$ otherwise.
	Since the algorithm is a descent method, $f(x_{N_2}) \leq f(x_{N_1}) \leq f(x_0)$.
\end{theorem}

\begin{remark}
	Theorem~\ref{thm:rtr-scnd-main} makes a statement about $\lambdamin(H_k)$ at termination, \emph{not} about $\lambdamin(\Hess f(x_k))$. See Section~\ref{sec:HandHess} to connect these two quantities.
\end{remark}

To establish Theorem~\ref{thm:rtr-scnd-main}, we work through a few lemmas, following the proof technique in~\citep{cartis2012complexity}. We first show $\Delta_k$ is bounded below in proportion to the tolerances $\varepsilon_g$ and $\varepsilon_H$. This is used to show that the number of successful iterations in Algorithm~\ref{algo:RTRscnd} before termination (that is, iterations where $\rho_k > \rho'$~\eqref{eq:rhok}) is bounded above. It is then shown that the total number of iterations is at most a constant multiple of the number of successful iterations, which implies termination in bounded time.
We start by showing that the trust-region radius is bounded away from zero.
Essentially, this is because if $\Delta_k$ becomes too small, then the Cauchy step and eigenstep are certain to be successful owing to the quality of the model in such a small region, so that the trust-region radius could not decrease any further.
\begin{lemma} \label{lemma:rtrscnd-delta-lowerbound}
	Under the assumptions of Theorem~\ref{thm:rtr-scnd-main},
	if Algorithm~\ref{algo:RTRscnd} executes $N$ iterations without terminating, then
	\begin{align}
		\Delta_k \geq \min\left( \Delta_0, \lambda_g \varepsilon_g, \lambda_H \varepsilon_H \right)
	\end{align}
	for $k = 0, \ldots, N$, where $\lambda_g$ and $\lambda_H$ are defined in~\eqref{eq:rtr-lambda-gH}.
\end{lemma}
\begin{proof}
	This follows essentially the proof of~\cite[Thm.\,7.4.2]{AMS08} which itself follows classical proofs~\citep{conn2000trust}. The core idea is to control $\rho_k$~\eqref{eq:rhok} close to~1, to show that there cannot be arbitrarily many trust-region radius reductions. The proof is in two parts.
	
	For the first part, assume $\|\grad f(x_k)\| > \varepsilon_g$. Then, consider the gap
	\begin{align}
		|\rho_k - 1| & = \left| \frac{\hat f_k(0_{x_k}) - \hat f_k(\eta_k)}{\hat m_k(0_{x_k}) - \hat m_k(\eta_k)} - 1 \right| 
		= \left| \frac{\hat m_k(\eta_k) - \hat f_k(\eta_k)}{\hat m_k(0_{x_k}) - \hat m_k(\eta_k)} \right|.
		\label{eq:rhokgap}
	\end{align}
	From \aref{assu:rtr-Cauchy-step}, we know the denominator is not too small:
	\begin{align*}
		\hat m_k(0_{x_k}) - \hat m_k(\eta_k) & \geq c_2 \min\left( \Delta_k, \frac{\varepsilon_g}{c_0} \right) \varepsilon_g.
	\end{align*}
	Now consider the numerator:
	\begin{align*}
		|\hat m_k(\eta_k) - \hat f_k(\eta_k)| & = \left| f(x_k) + \inner{\grad\,f(x_k)}{\eta_k} + \frac{1}{2}\inner{\eta_k}{H_k[\eta_k]} - \hat f_k(\eta_k) \right| \\
		& \leq \big| f(x_k) + \inner{\grad\,f(x_k)}{\eta_k} - \hat f_k(\eta_k) \big| + \frac{1}{2}\big| \inner{\eta_k}{H_k[\eta_k]} \big| \\
		& \leq \frac{1}{2}\left( \Lg + c_0 \right)\|\eta_k\|^2,
	\end{align*}
	where we used \aref{assu:rtr-regularity-first} for the first term, and \aref{assu:rtr-Hk-first} for the second term. Assume for the time being that $\Delta_k \leq \min\left( \frac{\varepsilon_g}{c_0}, \frac{c_2 \varepsilon_g}{\Lg + c_0} \right) = 4\lambda_g \varepsilon_g$. Then, using $\|\eta_k\| \leq \Delta_k$, it follows that
	\begin{align*}
		|\rho_k - 1| \leq \frac{1}{2} \frac{\Lg + c_0}{c_2  \min\left( \Delta_k, \frac{\varepsilon_g}{c_0} \right)\varepsilon_g} \Delta_k^2 \leq \frac{1}{2}\frac{\Lg+c_0}{c_2 \varepsilon_g} \Delta_k \leq \frac{1}{2}.
	\end{align*}
	Hence, $\rho_k \geq 1/2$, and by the mechanism of Algorithm~\ref{algo:RTRscnd}, it follows that $\Delta_{k+1} \geq \Delta_k$.

	For the second part, assume $\|\grad f(x_k)\| < \varepsilon_g$ and $\lambdamin(H_k) < -\varepsilon_H$. Then, by \aref{assu:rtr-eigenstep},
	\begin{align*}
		\hat m_k(0_{x_k}) - \hat m_k(\eta_k) & \geq c_3 \Delta_k^2 \varepsilon_H.
	\end{align*}
	Thus, by \aref{assu:rtr-regularity-second} and \aref{assu:rtr-Hk-second},
	\begin{align*}
		|\hat m_k(\eta_k) - \hat f_k(\eta_k)| & = \left| f(x_k) + \inner{\grad\,f(x_k)}{\eta_k} + \frac{1}{2}\inner{\eta_k}{H_k[\eta_k]} - \hat f_k(\eta_k) \right| \\
		& \leq \frac{\LH}{6} \|\eta_k\|^3 + \frac{1}{2} \left| \inner{\eta_k}{\big(\nabla^2 \hat f_k(0_{x_k}) - H_k\big)[\eta_k]} \right| \\
		\\		& \leq \frac{\LH + c_1}{6} \Delta_k^3.
	\end{align*}
	As previously, combine these observations into~\eqref{eq:rhokgap} to see that, if $\Delta_k \leq \frac{3c_3}{\LH + c_1}\varepsilon_H = 4\lambda_H \varepsilon_H$, then
	\begin{align}
		|\rho_k - 1| & \leq \frac{1}{2} \frac{\LH + c_1}{3c_3 \varepsilon_H} \Delta_k \leq \frac{1}{2}.
	\end{align}
	Again, this implies $\Delta_{k+1} \geq \Delta_k$.
	
	Now combine the two parts. We have established that, if $\Delta_k \leq 4\min\left(\lambda_g \varepsilon_g, \lambda_H\varepsilon_H\right)$, then $\Delta_{k+1} \geq \Delta_k$. To conclude the proof, consider the fact that Algorithm~\ref{algo:RTRscnd} cannot reduce the radius by more than $1/4$ in one step.
\end{proof}

By an argument similar to the one used for gradient methods, Lemma~\ref{lemma:rtrscnd-delta-lowerbound} implies an upper bound on the number of successful iterations required in Algorithm~\ref{algo:RTRscnd} to reach termination.
\begin{lemma}
	\label{lem:rtrscnd-success-bound}
	Under the assumptions of Theorem~\ref{thm:rtr-scnd-main},
	if Algorithm~\ref{algo:RTRscnd} executes $N$ iterations without terminating, define the set of \emph{successful steps} as
	$$
	S_N = \{ k \in \{0, \ldots, N\} : \rho_k > \rho' \} 	$$
	and let $U_N$ designate the unsuccessful steps, so that $S_N$ and $U_N$ form a partition of $\{0, \ldots, N\}$. Assume $\varepsilon_g \leq \Delta_0 / \lambda_g$. If $\varepsilon_H = \infty$, the number of successful steps obeys
	\begin{align}
		|S_N| \leq \frac{f(x_0) - f^*}{\rho' c_2 \lambda_g} \frac{1}{\varepsilon_g^2}.
	\end{align}
	Otherwise, if additionally $\varepsilon_g \leq \frac{c_2}{c_3} \frac{\lambda_H}{\lambda_g^2}$ and $\varepsilon_H \leq \frac{c_2}{c_3} \frac{1}{\lambda_g}$, we have the bound
	\begin{align}
		|S_N| \leq \frac{f(x_0) - f^*}{\rho' c_3 } \frac{1}{\min(\lambda_g\varepsilon_g, \lambda_H\varepsilon_H)^2 \varepsilon_H}.
	\end{align}
\end{lemma}
\begin{proof}
	The proof parallels~\cite[Lemma\,4.5]{cartis2012complexity}.
	Clearly, if $k\in U_N$, then $f(x_k) = f(x_{k+1})$. On the other hand, if $k\in S_N$, then $\rho_k \geq \rho'$~\eqref{eq:rhok}. Combine this with \aref{assu:rtr-Cauchy-step} and \aref{assu:rtr-eigenstep} to see that, for $k \in S_N$,
	\begin{align*}
		f(x_k) - f(x_{k+1}) & \geq \rho' \big(\hat m_k(0_{x_k}) - \hat m_k(\eta_k)\big) \\
		& \geq \rho' \min\left( c_2 \min\left( \Delta_k, \frac{\varepsilon_g}{c_0} \right) \varepsilon_g \ , \ c_3 \Delta_k^2 \varepsilon_H \right).
	\end{align*}
	By Lemma~\ref{lemma:rtrscnd-delta-lowerbound} and the assumption $\lambda_g\varepsilon_g \leq \Delta_0$, it holds that $\Delta_k \geq \min\left( \lambda_g \varepsilon_g, \lambda_H \varepsilon_H \right)$. Furthermore, using $\lambda_g \leq 1/c_0$ shows that $\min(\Delta_k, \varepsilon_g / c_0) \geq \min(\Delta_k, \lambda_g \varepsilon_g) \geq \min\left( \lambda_g \varepsilon_g, \lambda_H \varepsilon_H \right)$. Hence,
	\begin{align}
		f(x_k) - f(x_{k+1}) & \geq \rho' \min\left( c_2 \lambda_g \varepsilon_g^2, c_2 \lambda_H \varepsilon_g \varepsilon_H, c_3 \lambda_g^2 \varepsilon_g^2 \varepsilon_H, c_3 \lambda_H^2 \varepsilon_H^3 \right).
		\label{eq:foo123}
	\end{align}
	If $\varepsilon_H = \infty$, this simplifies to
	$$
	f(x_k) - f(x_{k+1}) \geq \rho' c_2 \lambda_g \varepsilon_g^2.
	$$
	Sum over iterations up to $N$ and use~\aref{assu:lowerbound} (bounded $f$):
	\begin{align*}
		f(x_0) - f^* \geq f(x_0) - f(x_{N+1}) & = \sum_{k\in S_N} f(x_k) - f(x_{k+1})
		\geq |S_N| \rho' c_2 \lambda_g \varepsilon_g^2.
	\end{align*}
	Hence,
	\begin{align*}
		|S_N| \leq \frac{f(x_0) - f^*}{\rho' c_2 \lambda_g} \frac{1}{\varepsilon_g^2}.
	\end{align*}
	On the other hand, if $\varepsilon_H < \infty$, then, starting over from~\eqref{eq:foo123} and assuming both  $c_3 \lambda_g^2 \varepsilon_g^2 \varepsilon_H \leq c_2 \lambda_H \varepsilon_g \varepsilon_H$ and $c_3 \lambda_g^2 \varepsilon_g^2 \varepsilon_H \leq c_2\lambda_g\varepsilon_g^2$ (which is equivalent to $\varepsilon_g \leq c_2\lambda_H / c_3\lambda_g^2$ and $\varepsilon_H \leq c_2/c_3\lambda_g$), it comes with the same telescoping sum that
	\begin{align*}
		f(x_0) - f^* \geq |S_N| \rho' c_3 \min(\lambda_g\varepsilon_g, \lambda_H\varepsilon_H)^2 \varepsilon_H.
	\end{align*}
	Solve for $|S_N|$ to conclude.
\end{proof}

Finally, we show that the total number of steps $N$ before termination cannot be more than a fixed multiple of the number of successful steps $|S_N|$.

\begin{lemma} \label{lem:rtrscnd-success-fraction-N}
	Under the assumptions of Theorem~\ref{thm:rtr-scnd-main},
	if Algorithm~\ref{algo:RTRscnd}  executes $N$ iterations without terminating, using the notation $S_N$ and $U_N$ of Lemma~\ref{lem:rtrscnd-success-bound}, it holds that
	\begin{align}
		|S_N| & \geq \frac{2}{3}(N+1) - \frac{1}{3}\max\left( 0 , \log_2\left( \frac{\Delta_0}{\lambda_g \varepsilon_g} \right) , \log_2\left( \frac{\Delta_0}{\lambda_H \varepsilon_H} \right) \right).
	\end{align}
\end{lemma}
\begin{proof}
	The proof rests on the lower bound for $\Delta_k$ obtained in Lemma~\ref{lemma:rtrscnd-delta-lowerbound}. It parallels~\cite[Lemma\,4.6]{cartis2012complexity}. For all $k\in S_N$, it holds that $\Delta_{k+1} \leq 2 \Delta_k$. For all $k\in U_k$, it holds that $\Delta_{k+1} \leq \frac{1}{4} \Delta_k$. Hence,
	\begin{align*}
		\Delta_N \leq 2^{|S_N|} \left(\frac{1}{4}\right)^{|U_N|} \Delta_0.
	\end{align*}
	On the other hand, Lemma~\ref{lemma:rtrscnd-delta-lowerbound} gives
	\begin{align*}
		\Delta_N \geq \min\left( \Delta_0, \lambda_g \varepsilon_g, \lambda_H \varepsilon_H \right).
	\end{align*}
	Combine, divide by $\Delta_0$ and take the log in base 2:
	\begin{align*}
		|S_N| - 2|U_N| \geq \min\left( 0, \log_2\left( \frac{\lambda_g \varepsilon_g}{\Delta_0}\right), \log_2\left( \frac{\lambda_H \varepsilon_H}{\Delta_0} \right) \right).
	\end{align*}
	Use $|S_N| + |U_N| = N+1$ to conclude.
\end{proof}

We can now prove the main theorem.
\begin{proof}[Proof of Theorem~\ref{thm:rtr-scnd-main}]
	It is sufficient to combine Lemmas~\ref{lem:rtrscnd-success-bound} and~\ref{lem:rtrscnd-success-fraction-N} in both regimes. First, we get that if $\|\grad f(x_{k})\| > \varepsilon_g$ for $k = 0, \ldots, N$, then
	\begin{align*}
		N+1 \leq \frac{3}{2} \frac{f(x_0) - f^*}{\rho' c_2 \lambda_g} \frac{1}{\varepsilon_g^2} + \frac{1}{2} \log_2\left( \frac{\Delta_0}{\lambda_g \varepsilon_g} \right).
	\end{align*}
	(The term $\log_2\left( \frac{\Delta_0}{\lambda_H \varepsilon_H} \right)$ from Lemma~\ref{lem:rtrscnd-success-fraction-N} is irrelevant up to that point, as $\varepsilon_H$ could just as well have been infinite.) Thus, after a number of iterations larger than the right hand side, an iterate with sufficiently small gradient must have been produced, to avoid a contradiction.
	
	Second, we get that if for $k = 0, \ldots, N$ no iterate satisfies both $\|\grad f(x_k)\| \leq \varepsilon_g$ and $\lambdamin(H_k) \geq -\varepsilon_H$, then
	\begin{align*}
		N+1 \leq \frac{3}{2} \frac{f(x_0) - f^*}{\rho' c_3 } \frac{1}{\min(\lambda_g\varepsilon_g, \lambda_H\varepsilon_H)^2 \varepsilon_H} + \frac{1}{2}\max\left( \log_2\left( \frac{\Delta_0}{\lambda_g \varepsilon_g} \right) , \log_2\left( \frac{\Delta_0}{\lambda_H \varepsilon_H} \right) \right).
	\end{align*}
	Conclude with the same argument.
\end{proof}

\subsection{Connecting $H_k$ and $\Hess f(x_k)$}\label{sec:HandHess}

Theorem~\ref{thm:rtr-scnd-main} states termination of Algorithm~\ref{algo:RTRscnd} in terms of $\|\grad f(x_k)\|$ and $\lambdamin(H_k)$. Ideally, the latter must be turned into a statement about $\lambdamin(\Hess f(x_k))$, to match the second-order necessary optimality conditions of~\eqref{eq:P} more closely (recall Proposition~\ref{prop:necessaryconditions}). \aref{assu:rtr-Hk-second} itself only requires $H_k$ to be (weakly) related to $\nabla^2 \hat f_k(0_{x_k})$ (the Hessian of the pullback of $f$ at $x_k$), which is different from the Riemannian Hessian of $f$ at $x_k$ in general.
It is up to the user to provide $H_k$ sufficiently related to $\nabla^2 \hat f_k(0_{x_k})$. Additional control over the retraction at $x_k$ can further relate $\nabla^2 \hat f_k(0_{x_k})$ to $\Hess f(x_k)$, as we do now. Proofs for this section are in Appendix~\ref{apdx:Hfirstretr}.
\begin{lemma} \label{lem:scndretr}
		Define the maximal acceleration of $\Retr$ at $x$ as the real $a$ such that
	\begin{align*}
	\forall \eta \in \T_x\calM \textrm{ with } \|\eta\| = 1, \quad \Big. \left\| \frac{\mathrm{D}^2}{\mathrm{d}t^2} \Retr_x(t\eta) \Big|_{t=0} \right\| \leq a, 	\end{align*}
	where $\frac{\mathrm{D}^2}{\mathrm{d}t^2} \gamma $ denotes acceleration of the curve $t \mapsto \gamma(t)$ on $\calM$~\cite[\S5]{AMS08}.
	Then, 
	\begin{align*}
	\left\| \Hess f(x) - \nabla^2 \hat f_x(0_x) \right\|
		\leq a \cdot \|\grad f(x)\|.
	\end{align*}
	In particular, if $x$ is a critical point or if $a = 0$, the Hessians agree: $\Hess f(x) = \nabla^2 \hat f_x(0_x)$.
\end{lemma}
The particular cases appear as~\citep[Prop.\,5.5.5, 5.5.6]{AMS08}. This result highlights the crucial role of retractions with zero acceleration, known as \emph{second-order retractions} and defined in~\citep[Prop.\,5.5.5]{AMS08}; we are not aware of earlier references to this notion.
\begin{definition}\label{def:retraction2}
	A retraction is a \emph{second-order retraction} if it has zero acceleration, as defined in Lemma~\ref{lem:scndretr}. Then, retracted curves locally agree with geodesics up to second order.
\end{definition}
\begin{proposition}\label{prop:relateHandHess}
	Let $x_k\in\calM$ be the iterate returned by Algorithm~\ref{algo:RTRscnd} under the assumptions of Theorem~\ref{thm:rtr-scnd-main}. It satisfies $\|\grad f(x_k)\| \leq \varepsilon_g$ and $H_k \succeq -\varepsilon_H \Id$. Assume $H_k$ is related to the Hessian of the pullback as $\left\|\nabla^2 \hat f_k(0_{x_k}) - H_k\right\| \leq \delta_k$. Further assume the retraction has acceleration at $x_k$ bounded by $a_k$, as defined in Lemma~\ref{lem:scndretr}. Then,
	\begin{align*}
		\Hess f(x_k) \succeq -\left( \varepsilon_H + a_k \varepsilon_g + \delta_k \right) \Id.
	\end{align*}
	In particular, if the retraction is second-order and $H_k = \nabla^2 \hat f_k(0_{x_k})$, then $\Hess f(x_k) \succeq -\varepsilon_H\Id$.
\end{proposition}
We note that second-order retractions are frequently available in applications. Indeed, retractions for submanifolds obtained as (certain types of) projections---arguably one of the most natural classes of retractions for submanifolds---are second order~\citep[Thm.\,22]{absil2012retractions}. For example, the sphere retraction $\Retr_x(\eta) = (x+\eta)/\|x+\eta\|$ is second order. Such retractions for low-rank matrices are also known~\citep{absil2015lowrankretractions}.

\section{Example: smooth semidefinite programs}\label{sec:example}

This example is based on~\citep{boumal2016bmapproach}. Consider the  following semidefinite program, which occurs in robust PCA~\citep{mccoy2011robustpca} and as a convex relaxation of combinatorial problems such as Max-Cut, $\mathbb{Z}_2$-synchronization and community detection in the stochastic block model~\citep{goemans1995maxcut,bandeira2016lowrankmaxcut}:
\begin{align}
	\min_{X\in\Rnn} \trace(CX) \textrm{ subject to } \diag(X) = \mathbf{1}, X \succeq 0.
	\label{eq:SDP}
\end{align}
The symmetric cost matrix $C$ depends on the application. Interior point methods solve this problem in polynomial time, though they involve significant work to enforce the conic constraint $X\succeq 0$ ($X$ symmetric, positive semidefinite).
This motivates the approach of~\citet{burer2005local} to parameterize the search space as $X = YY\transpose$, where $Y$ is in $\Rnp$ for some well-chosen $p$:
\begin{align}
	\min_{Y\in\Rnp} \trace(CYY\transpose) \textrm{ subject to } \diag(YY\transpose) = \mathbf{1}.
	\label{eq:BM}
\end{align}
This problem is of the form of~\eqref{eq:P}, where $f(Y) = \trace(CYY\transpose)$ and the manifold is a product of $n$ unit spheres in $\Rp$:
\begin{align}
	\calM & = \{ Y\in\Rnp : \diag(YY\transpose) = \mathbf{1} \} = \{ Y\in\Rnp : \textrm{each row of $Y$ has unit norm} \}.
	\label{eq:calMmaxcut}
\end{align}
In principle, since the parameterization $X=YY\transpose$ breaks convexity, the new problem could have many spurious local optimizers and saddle points. Yet, for $p = n+1$, it has recently been shown that approximate second-order critical points $Y$ map to approximate global optimizers $X=YY\transpose$, as stated in the following proposition. (In this particular case, there is no need to control $\|\grad f(Y)\|$ explicitly.)
\begin{proposition}[\citet{boumal2016bmapproach}]\label{prop:maxcutapprox}
	If $X^\star$ is optimal for~\eqref{eq:SDP} and $Y$ is feasible for~\eqref{eq:BM} with $p>n$ and $\Hess f(Y) \succeq -\varepsilon_H\Id$, the optimality gap is bounded as
	\begin{align*}
		0 \leq \trace(CYY\transpose) - \trace(CX^\star) \leq \frac{n}{2}\varepsilon_H.
	\end{align*}
\end{proposition}
Since $f$ is smooth in $\Rnp$ and $\calM$ is a compact submanifold of $\Rnp$, the regularity assumptions~\aref{assu:rtr-regularity-first} and~\aref{assu:rtr-regularity-second} hold with any second-order retraction (Lemmas~\ref{lem:compactsub1} and~\ref{lem:compactsub2}). In particular, they hold if $\Retr_Y(\dot Y)$ is the result of normalizing each row of $Y+\dot Y$ (Section~\ref{sec:HandHess}), or if the exponential map is used (which is cheap for this manifold, see Appendix~\ref{apdx:examplecomplexity}). Theorem~\ref{thm:rtr-scnd-main} then implies that RTR applied to the nonconvex problem~\eqref{eq:BM} computes a point $X = YY\transpose$ feasible for~\eqref{eq:SDP} such that $\Trace(CX)-\Trace(CX^\star) \leq \delta$ in $\calO(1/\delta^3)$ iterations. Appendix~\ref{apdx:examplecomplexity} bounds the total work with an explicit dependence on the problem dimension $n$ as $\calO(n^{10} / \delta^3)$ arithmetic operations, where $\calO$ hides factors depending on the data $C$ and an additive log-term. This result follows from a bound $\LH \leq 8\opnorm{C} \sqrt{n}$ for~\aref{assu:rtr-regularity-second} which is responsible for a factor of $n$ in the complexity---the remaining factors could be improved, see below.

In~\citep{boumal2016bmapproach}, it is shown that, generically in $C$, if $p \geq \lceil\sqrt{2n}\rceil$, then all second-order critical points of~\eqref{eq:BM} are globally optimal (despite nonconvexity). This means RTR globally converges to global optimizers with cheaper iterations (due to reduced dimensionality). Unfortunately, there is no statement of quality pertaining to \emph{approximate} second-order critical points for such small $p$, so that this analysis is not sufficient to obtain an improved worst-case complexity bound.

These bounds are worse than guarantees provided by interior point methods. Indeed, following~\citep[\S4.3.3, with eq.\,(4.3.12)]{nesterov2004introductory}, interior point methods achieve a solution in $\calO(n^{3.5}\log(n/\delta))$ arithmetic operations.
Yet, numerical experiments in~\citep{boumal2016bmapproach} suggest RTR often outperforms interior point methods, indicating the bound $\calO(n^{10}/\delta^3)$ is wildly pessimistic. We report it here mainly because, to the best of our knowledge, this is the first explicit bound for a Burer--Monteiro approach to solving a semidefinite program.

A number of factors drive the gap between our worst-case bound and practice. In particular, strategies far more efficient than the $LDL\transpose$ factorization in Lemma~\ref{lem:EigenStep} are used to compute second-order steps, and they can exploit structure in $C$. High accuracy solutions are reached owing to RTR typically converging superlinearly, locally. And $p$ is chosen much smaller than $n+1$.

See also~\citep{mei2017solvingSDPs} for formal complexity results in a setting where $p$ is allowed to be independent of $n$; this precludes reaching an objective value arbitrarily close to optimal, in exchange for cheaper computations.

\section{Conclusions and perspectives}\label{sec:conclusions}

We presented bounds on the number of iterations required by the Riemannian gradient descent algorithm and the Riemannian trust-region algorithm to reach points which approximately satisfy first- and second-order necessary optimality conditions, under some regularity assumptions but regardless of initialization. When the search space $\calM$ is a Euclidean space, these bounds were already known. For the more general case of $\calM$ being a Riemannian manifold, these bounds are new.

As a subclass of interest, we showed the regularity requirements are satisfied if $\calM$ is a compact submanifold of $\Rn$ and $f$ has locally Lipschitz continuous derivatives of appropriate order. This covers a rich class of practical optimization problems. 
While there are no explicit assumptions made about $\calM$, the smoothness requirements for the pullback of the cost---\aref{assu:lipschitzpullbackrestricted}, \aref{assu:rtr-regularity-first} and \aref{assu:rtr-regularity-second}---implicitly restrict the class of manifolds to which these results apply. Indeed, for certain manifolds, even for nice cost functions $f$, there may not exist retractions which ensure the assumptions hold. This is the case in particular for certain incomplete manifolds, such as open Riemannian submanifolds of $\Rn$ and certain geometries of the set of fixed-rank matrices---see also Remark~\ref{rem:domainrestriction} about injectivity radius. For such sets, it may be necessary to adapt the assumptions. For fixed-rank matrices for example, \citet[\S4.1]{vandereycken2013lowrank} obtains convergence results assuming a kind of coercivity on the cost function: for any sequence of rank-$k$ matrices $(X_i)_{i=1,2,\ldots}$ such that the first singular value $\sigma_1(X_i) \to \infty$ or the $k$th singular value $\sigma_{k}(X_i) \to 0$, it holds that $f(X_i) \to \infty$. This ensures iterates stay away from the open boundary.

The iteration bounds are sharp, but additional information may yield
more favorable bounds in specific contexts. In particular, when the
studied algorithms converge to a nondegenerate local optimizer, they do
so with an at least linear rate, so that the number of iterations is merely
$\calO(\log(1/\varepsilon))$ once in the linear regime.
This suggests a stitching approach: for a given application, it may be possible to show that rough approximate second-order critical points are in a local attraction basin; the iteration cost can then be bounded by the total work needed to attain such a crude point starting from anywhere, plus the total work needed to refine the crude point to high accuracy with a linear or even quadratic convergence rate.
This is, to some degree, the successful strategy in~\citep{sun2015complete,sun2016geometric}.

Finally, we note that it would also be interesting to study the global convergence rates of Riemannian versions of adaptive regularization algorithms using cubics (ARC), as in the Euclidean case these can achieve approximate first-order criticality in $\calO(1/\varepsilon^{1.5})$ instead of $\calO(1/\varepsilon^2)$~\citep{cartis2011adaptivecubic}. Work in that direction could start with the convergence analyses proposed in~\citep{qi2011thesis}.

\section*{Acknowledgments}
\begin{footnotesize}
NB was supported by the ``Fonds Sp\'eciaux de Recherche'' (FSR) at UCLouvain and by the Chaire Havas ``Chaire Eco\-no\-mie et gestion des nouvelles don\-n\'ees'', the ERC Starting Grant SIPA and a Research in Paris grant at Inria \& ENS, and NSF DMS-1719558.
This paper presents research results of the Belgian Network DYSCO (Dynamical Systems, Control, and Optimization), funded by the Interuniversity Attraction Poles Programme initiated by the Belgian Science Policy Office. This work was supported by the ARC
``Mining and Optimization of Big Data Models''. CC acknowledges support from NERC through grant NE/L012146/1.
We thank Alex d'Aspremont, Simon Lacoste-Julien, Ju Sun, Bart Vandereycken and Paul Van Dooren for helpful discussions.
\end{footnotesize}

\bibliographystyle{abbrvnat}
\bibliography{../../boumal}

\appendix

\section{Essentials about manifolds}\label{apdx:manifolds}

We give here a simplified refresher of differential geometric concepts used in the paper, restricted to Riemannian submanifolds. All concepts are illustrated with the sphere.  See~\citep{AMS08} for a more complete discussion, including quotient manifolds.

We endow $\Rn$ with the classical Euclidean metric: for all $x, y \in \Rn$, $\inner{x}{y} = x\transpose y$. Consider the smooth map $h \colon \Rn \mapsto \Rm$ with $m \leq n$ and the constraint set
\begin{align*}
\calM & = \{ x \in \Rn : h(x) = 0 \}.
\end{align*}
Locally around each $x$, this set can be linearized by differentiating the constraint. The subspace corresponding to this linearization is the kernel of the differential of $h$ at $x$~\citep[eq.~(3.19)]{AMS08}:
\begin{align*}
\T_x\calM & = \{ \eta \in \Rn : \D h(x)[\eta] = 0 \}.
\end{align*}
If this subspace has dimension $n-m$ for all $x\in\calM$, then $\calM$ is a submanifold of dimension $n-m$ of $\Rn$~\citep[Prop.~3.3.3]{AMS08} and $\T_x\calM$ is called the tangent space to $\calM$ at $x$.
For example, the unit sphere in $\Rn$ is a submanifold of dimension $n-1$ defined by
\begin{align*}
\mathcal{S}^{n-1} & = \{ x \in \Rn : x\transpose x = 1 \},
\end{align*}
and the tangent space at $x$ is
\begin{align*}
\T_x\mathcal{S}^{n-1} & = \{ \eta \in \Rn : x\transpose \eta = 0 \}.
\end{align*}
By endowing each tangent space with the (restricted) Euclidean metric, we turn $\calM$ into a Riemannian submanifold of the Euclidean space $\Rn$. (In general, the metric could be different, and would depend on $x$; to disambiguate, one would write $\inner{\cdot}{\cdot}_x$.) An obvious retraction for the sphere (see Definition~\ref{def:retraction}) is to normalize:
\begin{align*}
\Retr_x(\eta) & = \frac{x+\eta}{\|x+\eta\|}.
\end{align*}
Being an orthogonal projection to the manifold, this is actually a second-order retraction, see Definition~\ref{def:retraction2} and~\citep[Thm.\,22]{absil2012retractions}.

The Riemannian metric leads to the notion of Riemannian gradient of a real function~$f$ defined in an open set of $\Rn$ containing $\calM$.\footnote{$f$ needs not be defined outside of $\calM$, but this is often the case in applications and simplifies exposition.} The Riemannian gradient of $f$ at $x$ is the (unique) tangent vector $\grad f(x)$ at $x$ satisfying
\begin{align*}
\forall \eta \in \T_x\calM, \quad \D f(x)[\eta] = \lim_{t \to 0} \frac{f(x+t\eta)-f(x)}{t} = \inner{\eta}{\grad f(x)}.
\end{align*}
In this setting, the Riemannian gradient is nothing but the orthogonal projection of the Euclidean (classical) gradient $\nabla f(x)$ to the tangent space. Writing $\Proj_x \colon \Rn \to \T_x\calM$ for the orthogonal projector, we have~\cite[eq.~(3.37)]{AMS08}:
\begin{align*}
\grad f(x) & = \Proj_x\!\left( \nabla f(x) \right).
\end{align*}
Continuing the sphere example, the orthogonal projector is $\Proj_x(y) = y - (x\transpose y) x$, and if $f(x) = \frac{1}{2} x\transpose A x$ for some symmetric matrix $A$, then
\begin{align*}
\nabla f(x) & = Ax, & \textrm{ and} & & \grad f(x) & = Ax - (x\transpose A x) x.
\end{align*}
Notice that the critical points of $f$ on $\mathcal{S}^{n-1}$ coincide with the unit eigenvectors of $A$.

We can further define a notion of Riemannian Hessian as the projected differential of the Riemannian gradient:\footnote{Proper definition of Riemannian Hessians requires the notion of Riemannian connections, which we omit here; see~\cite[\S5]{AMS08}}
\begin{align*}
\Hess f(x)[\eta] & = \Proj_x \Big( \D\big( x \mapsto \Proj_x \nabla f(x) \big)(x)[\eta]\Big).
\end{align*}
$\Hess f(x)$ is a linear map from $\T_x\calM$ to itself, symmetric with respect to the Riemannian metric. Given a second-order retraction (Definition~\ref{def:retraction2}), it is equivalently defined by:
\begin{align*}
\forall \eta \in \T_x\calM, \quad \inner{\eta}{\Hess f(x)[\eta]} & = \left. \frac{\mathrm{d}^2}{\mathrm{d}t^2} f(\Retr_x(t\eta))\right|_{t=0},
\end{align*}
see~\citep[eq.\,(5.35)]{AMS08}. Continuing our sphere example,
\begin{align*}
\D\big( x \mapsto \Proj_x \nabla f(x) \big)(x)[\eta] & = \D\big( x \mapsto Ax - (x\transpose Ax)x \big)(x)[\eta] = A\eta - (x\transpose Ax)\eta - 2 (x\transpose A \eta) x.
\end{align*}
Projection of the latter gives the Hessian:
\begin{align*}
\Hess f(x)[\eta] & = \Proj_x (A\eta) - (x\transpose Ax)\eta.
\end{align*}
Consider the implications of a positive semidefinite Hessian (on the tangent space):
\begin{align*}
\Hess f(x) \succeq 0 & \iff \inner{\eta}{\Hess f(x)[\eta]} \geq 0 & & \forall \eta\in\T_x\mathcal{S}^{n-1} \\
& \iff \eta\transpose A\eta  \geq x\transpose A x & & \forall \eta\in\T_x\mathcal{S}^{n-1}, \|\eta\| = 1.
\end{align*}
Together with first-order conditions, this implies that $x$ is a leftmost eigenvector of~$A$.\footnote{Indeed, any $y\in\mathcal{S}^{n-1}$ can be written as $y = \alpha x + \beta \eta$ with $x\transpose \eta = 0$, $\|\eta\| = 1$ and $\alpha^2 + \beta^2 = 1$; then, $y\transpose A y = \alpha^2 x\transpose A x + \beta^2 \eta\transpose A \eta + 2\alpha\beta \eta\transpose A x$; by first-order condition, $\eta\transpose Ax = (x\transpose A x) \eta\transpose x = 0$, and by second-order condition: $y\transpose A y \geq (\alpha^2+\beta^2)x\transpose A x = x\transpose A x$, hence $x\transpose A x$ is minimal over $\mathcal{S}^{n-1}$.}
This is an example of optimization problem on a manifold for which second-order necessary optimality conditions are also sufficient. This is not the norm.

As another (very) special example, consider the case $\calM = \Rn$; then, $\T_x\Rn = \Rn$, $\Retr_x(\eta) = x+\eta$ is the exponential map (a fortiori a second-order retraction), $\Proj_x$ is the identity, $\grad f(x) = \nabla f(x)$ and $\Hess f(x) = \nabla^2 f(x)$.

\section{Compact submanifolds of Euclidean spaces} \label{apdx:submanifold}

In this appendix, we prove Lemmas~\ref{lem:compactsub1} and~\ref{lem:compactsub2}, showing that if $f$ has locally Lipschitz continuous gradient or Hessian in a Euclidean space $\calE$ (in the usual sense), and it is to be minimized over a compact submanifold of $\calE$, then
\aref{assu:lipschitzpullbackrestricted},~\aref{assu:rtr-regularity-first} and~\aref{assu:rtr-regularity-second} hold.

\begin{proof}[\protect{Proof of Lemma~\ref{lem:compactsub1}}]
	By assumption, $\nabla f$ is Lipschitz continuous along any line segment in $\calE$ joining $x$ and $y$ in $\calM$. Hence,
			there exists $L$ such that, for all $x, y \in \calM$,	
	\begin{align}
	\big| f(y) - \left[ f(x) + \inner{\nabla f(x)}{y-x} \right] \big| \leq \frac{L}{2} \|y-x\|^2.
	\label{eq:compactsubeqfoo}
	\end{align}
	In particular, this holds for all $y = \Retr_x(\eta)$, for any $\eta \in \T_x\calM$. Writing $\grad f(x)$ for the Riemannian gradient of $f|_\calM$ and using that $\grad f(x)$ is the orthogonal projection of $\nabla f(x)$ to $\T_x\calM$~\cite[eq.~(3.37)]{AMS08}, the inner product above decomposes as
	\begin{align}
	\inner{\nabla f(x)}{\Retr_x(\eta)-x} & = \inner{\nabla f(x)}{\eta + \Retr_x(\eta) - x - \eta} \nonumber\\
	& = \inner{\grad f(x)}{\eta} + \inner{\nabla f(x)}{\Retr_x(\eta) - x - \eta}.
	\label{eq:compactsubeqbar}
	\end{align}
	Combining~\eqref{eq:compactsubeqfoo} with~\eqref{eq:compactsubeqbar} and using the triangle inequality yields
	\begin{align*}
	\big| f(\Retr_x(\eta)) - \left[ f(x) + \inner{\grad f(x)}{\eta} \right] \big| & \leq \frac{L}{2} \|\Retr_x(\eta)-x\|^2 + \|\nabla f(x)\| \|\Retr_x(\eta) - x - \eta\|.
	\end{align*}
	Since $\nabla f(x)$ is continuous on the compact set $\calM$, there exists $G$ finite such that $\|\nabla f(x)\| \leq G$ for all $x\in\calM$. It remains to show there exist finite constants $\alpha, \beta \geq 0$ such that, for all $x\in\calM$ and for all $\eta \in \T_x\calM$,
	\begin{align}
	\|\Retr_x(\eta)-x\| & \leq \alpha\|\eta\|, \textrm{ and} \label{eq:Retr0} \\
	\|\Retr_x(\eta) - x - \eta\| & \leq \beta\|\eta\|^2. \label{eq:Retr1}
	\end{align}
	For small $\eta$, this will follow from $\Retr_x(\eta) = x + \eta + \calO(\|\eta\|^2)$ by Definition~\ref{def:retraction}; for large $\eta$ this will follow a fortiori from compactness. This will be sufficient to conclude, as then we will have for all $x\in\calM$ and $\eta\in\T_x\calM$ that
	\begin{align*}
	\big| f(\Retr_x(\eta)) - \left[ f(x) + \inner{\grad f(x)}{\eta} \right] \big| & \leq \left(\frac{L}{2} \alpha^2 + G \beta \right)\|\eta\|^2.
	\end{align*}
	More formally, our assumption that the retraction is defined and smooth over the whole tangent bundle a fortiori ensures the existence of $r>0$ such that $\Retr$ is smooth on $K = \{\eta\in \T\calM: \|\eta\| \leq r\}$, a compact subset of the tangent bundle ($K$ consists of a ball in each tangent space). First, we determine $\alpha$~\eqref{eq:Retr0}.
		For all $\eta\in K$, we have
	\begin{align*}
	\|\Retr_x(\eta) - x \| & \leq \int_0^1 \left\|\ddt \Retr_x(t\eta)\right\| \dt
	= \int_0^1
	\|\D\Retr_x(t\eta)[\eta]\| \dt \\
	& \leq \int_0^1 \max_{\xi\in K} \|\D\Retr(\xi)\| \|\eta\| \dt
	= \max_{\xi\in K} \|\D\Retr(\xi)\| \|\eta\|,
	\end{align*}
	where the $\max$ exists and is finite owing to compactness of $K$ and smoothness of $\Retr$ on $K$; note that this is uniform over both $x$ and $\eta$. (If $\xi\in\T_z\calM$, the notation $\D\Retr(\xi)$ refers to $\D\Retr_z(\xi)$.)
		For all $\eta\notin K$, we have
	\begin{align*}
	\|\Retr_x(\eta) - x \| \leq \mathrm{diam}(\calM) \leq
	\frac{\mathrm{diam}(\calM)}{r} \|\eta\|,
	\end{align*}
	where $\mathrm{diam}(\calM)$ is the maximal distance between any two points on $\calM$: finite by compactness of $\calM$. Combining, we find that~\eqref{eq:Retr0} holds with
	\begin{align*}
	\alpha = \max\left( \max_{\xi\in K} \|\D\Retr(\xi)\|, \frac{\mathrm{diam}(\calM)}{r} \right).
	\end{align*}
			Inequality~\eqref{eq:Retr1} is established along similar lines. For all $\eta\in K$, we have
	\begin{align*}
	\|\Retr_x(\eta) - x - \eta\| & \leq \int_0^1 \left\|\ddt (\Retr_x(t\eta)-x-t\eta) \right\|
	\dt = \int_0^1 \|\D\Retr_x(t\eta)[\eta] - \eta\| \dt \\
	& \leq \int_0^1 \|\D\Retr_x(t\eta)-\Id\| \|\eta\| \dt
	\leq \frac{1}{2} \max_{\xi\in K} \|\D^2\Retr(\xi)\| \|\eta\|^2,
	\end{align*}
	where the last inequality follows from $\D\Retr_x(0_x) = \Id$ and 
	\begin{align*}
	\|\D\Retr_x(t\eta)-\Id\| & \leq \int_0^1 \left\| \dds \D\Retr_x(st\eta) \right\|  \ds \leq  \|t\eta\| \int_0^1 \left\| \D^2\Retr_x(t\eta) \right\|  \ds.
	\end{align*}
	The case $\eta\notin K$ is treated as before:
	\begin{align*}
	\|\Retr_x(\eta) - x - \eta \| \leq \|\Retr_x(\eta) - x\| + \|\eta \|
		\leq \frac{\mathrm{diam}(\calM) + r}{r^2}  \|\eta\|^2.
	\end{align*}
	Combining, we find that~\eqref{eq:Retr1} holds with
	\begin{align*}
	\beta & = \max\left( \frac{1}{2}  \max_{\xi\in K} \|\D^2\Retr(\xi)\| , \frac{\mathrm{diam}(\calM) + r}{r^2}  \right),
	\end{align*}
	which concludes the proof.
																		\end{proof}

We now prove the corresponding second-order result, whose aim is to verify~\aref{assu:rtr-regularity-second}.

\begin{proof}[\protect{Proof of Lemma~\ref{lem:compactsub2}}]
	By assumption, $\nabla^2 f$ is Lipschitz continuous along any line segment in $\calE$ joining $x$ and $y$ in $\calM$. Hence, there exists $L$ such that, for all $x, y \in \calM$,	
	\begin{align}
	\left| f(y) - \left[ f(x) + \inner{\nabla f(x)}{y-x} + \frac{1}{2}\inner{y-x}{\nabla^2 f(x)[y-x]} \right] \right| \leq \frac{L}{6} \|y-x\|^3.
	\label{eq:compactsub2base}
	\end{align}
	Fix $x\in\calM$. Let $\Proj_x$ denote the orthogonal projector from $\calE$ to $\T_x\calM$. Let $\grad f(x)$ be the Riemannian gradient of $f|_\calM$ at $x$ and let $\Hess f(x)$ be the Riemannian Hessian of $f|_\calM$ at $x$ (a symmetric operator on $\T_x\calM$). For Riemannian submanifolds of Euclidean spaces we have these explicit expressions with $\eta\in\T_x\calM$---see~\citep[eqs.~(3.37), (5.15), Def.~(5.5.1)]{AMS08} and~\citep{absil2013extrinsic}:
	\begin{align*}
	\grad f(x) & = \Proj_x \nabla f(x), \textrm{ and} \\
	\inner{\eta}{\Hess f(x)[\eta]} & = \inner{\eta}{ \D\big( x \mapsto \Proj_x \nabla f(x) \big)(x)[\eta]} \\
	& = \inner{\eta}{\Big( \D\big( x \mapsto \Proj_x \big)(x)[\eta]\Big)[\nabla f(x)] + \Proj_x \nabla^2f(x)[\eta]} \\
	& = \inner{\II(\eta,\eta)}{\nabla f(x)} + \inner{\eta}{\nabla^2 f(x)[\eta]},
	\end{align*}
	where $\II$, as implicitly defined above, is the \emph{second fundamental form} of $\calM$: $\II(\eta, \eta)$ is a normal vector to the tangent space at $x$, capturing the second-order geometry of $\calM$---see~\citep{absil2009roadsnewton,absil2013extrinsic,monera2014taylorexp} for presentations relevant to our setting. In particular, $\II(\eta,\eta)$ is the acceleration in $\calE$ at $x$ of a geodesic $\gamma(t)$ on $\calM$ defined by $\gamma(0) = x$ and $\dot\gamma(0) = \eta$: $\ddot\gamma(0) = \II(\eta,\eta)$~\citep[Cor.~4.9]{oneill}.
	
	Let $\eta \in \T_x\calM$ be arbitrary; $y = \Retr_x(\eta) \in \calM$. Then,
	\begin{align*}
	\inner{\nabla f(x)}{y-x} - \inner{\grad f(x)}{\eta} & = \inner{\nabla f(x)}{y - x - \eta} \textrm{ and} \\
	\inner{y-x}{\nabla^2 f(x)[y-x]} - \inner{\eta}{\Hess f(x)[\eta]} & = 2\inner{\eta}{\nabla^2 f(x)[y-x-\eta]} \\ & + \inner{y-x-\eta}{\nabla^2 f(x)[y-x-\eta]} \\ & - \inner{\nabla f(x)}{\II(\eta,\eta)}.
	\end{align*}
	Since $\calM$ is compact and $f$ is twice continuously differentiable, there exist $G, H$, independent of $x$, such that $\|\nabla f(x)\| \leq G$ and $\|\nabla^2 f(x)\| \leq H$ (the latter is the induced operator norm). Combining with~\eqref{eq:compactsub2base} and using the triangle and Cauchy--Schwarz inequalities multiple times,
	\begin{multline*}
	\left| f(y) - \left[ f(x) + \inner{\grad f(x)}{\eta} + \frac{1}{2}\inner{\eta}{\Hess f(x)[\eta]} \right] \right| \\ \leq \frac{L}{6} \|y-x\|^3 + G\left\|y-x-\eta-\frac{1}{2}\II(\eta,\eta)\right\| + H\|\eta\|\|y-x-\eta\| + \frac{1}{2} H \|y-x-\eta\|^2.
	\end{multline*}
	Using the same argument as for Lemma~\ref{lem:compactsub1}, we can find finite constants $\alpha,\beta$ independent of $x$ and $\eta$ such that~\eqref{eq:Retr0} and~\eqref{eq:Retr1} hold. Use $\|y-x-\eta\|^2 \leq \|y-x-\eta\| \left( \|y-x\| + \|\eta\| \right) \leq \beta(\alpha+1)\|\eta\|^3$ to upper bound the right hand side above with
	\begin{align*}
	\left(\frac{L}{6}\alpha^3 + H\beta + \frac{H\beta(\alpha+1)}{2}\right) \|\eta\|^3 + G\left\|y-x-\eta-\frac{1}{2}\II(\eta,\eta)\right\|.
	\end{align*}
	We turn to the last term. Consider $K \subset \T\calM$ as defined in  the proof of Lemma~\ref{lem:compactsub1} for some $r>0$. If $\eta \notin K$, i.e., $\|\eta\| > r$, then, since $\II$ is bilinear for a fixed $x\in\calM$, we can define
	\begin{align*}
		\|\II\| = \max_{x \in \calM} \max_{\xi \in \T_x\calM, \|\xi\| \leq 1} \|\II(\xi, \xi)\|
	\end{align*}
	(finite by continuity and compactness) so that $\|\II(\eta, \eta)\| \leq \|\II\| \|\eta\|^2$.
	Then,
	\begin{align*}
		\left\|y-x-\eta-\frac{1}{2}\II(\eta,\eta)\right\| & \leq \|y-x\| + \|\eta\| + \frac{1}{2} \|\II(\eta, \eta)\| \leq \left( \frac{\operatorname{diam}(\calM)}{r^3} + \frac{1}{r^2} + \frac{1}{2} \frac{\|\II\|}{r} \right) \|\eta\|^3.
	\end{align*}
	Now assume $\eta \in K$, that is, $\|\eta\| \leq r$. Consider $\phi(t) = \Retr_x(t\eta)$ (a curve on $\calM$) and let $\phi''$ denote its acceleration on $\calM$ and $\ddot \phi$ denote its acceleration in $\calE$, while $\dot \phi = \phi'$ denotes velocity along the curve. It holds that $\ddot \phi(t) = \phi''(t) + \II(\dot \phi(t), \dot \phi(t))$~\citep[Cor.~4.9]{oneill}. Since $\Retr$ is a second-order retraction, acceleration on $\calM$ is zero at $t=0$, i.e., $\phi''(0) = 0$, 	so that $\phi(0) = x$, $\dot \phi(0) = \eta$ and $\ddot \phi(0) = \II(\eta, \eta)$. Then, by Taylor expansion of $\phi$ in $\calE$,
	\begin{align*}
	y = \Retr_x(\eta) = \phi(1) = x + \eta + \frac{1}{2} \II(\eta,\eta) + R_3(\eta), 	\end{align*}
	where
	\begin{align*}
		\| R_3(\eta) \| & = \left\| \int_{0}^{1} \frac{(1-t)^2}{2} \dddot \phi(t) \dt \right\| \leq \frac{1}{6} \max_{\xi\in K} \| \D^3 \Retr(\xi) \| \|\eta\|^3.
	\end{align*}
    The combined arguments ensure existence of a constant $\gamma$, independent of $x$ and $\eta$, such that
	\begin{align*}
	\left\|y-x-\eta-\frac{1}{2}\II(\eta,\eta)\right\| \leq \gamma\|\eta\|^3.
	\end{align*}
	Combining, we find that for all $x\in\calM$ and $\eta\in\T_x\calM$,
	\begin{align*}
	\left| f(\Retr_x(\eta)) - \left[ f(x) + \inner{\grad f(x)}{\eta} + \frac{1}{2}\inner{\eta}{\Hess f(x)[\eta]} \right] \right| & \leq \left(\frac{L}{6}\alpha^3 + \frac{H\beta(\alpha+3)}{2} + \gamma\right) \|\eta\|^3.
	\end{align*}
	Since $\Retr$ is a second-order retraction, $\Hess f(x)$ coincides with the Hessian of the pullback $f\circ\Retr_x$ (Lemma~\ref{lem:scndretr}). This establishes~\aref{assu:rtr-regularity-second}.
\end{proof}

\section{Proof of Lemma~\ref{lem:armijo} about Armijo line-search} \label{apdx:lem:armijo}

\begin{proof}[\protect{Proof of Lemma~\ref{lem:armijo}}]
	By \aref{assu:lipschitzpullbackrestricted}, upper bound~\eqref{eq:upperboundpullback} holds with $\eta = t\eta_k^0$ for any $t$ such that $\|\eta\| \leq \varrho_k$:
	\begin{align}
	f(x_k) - f(\Retr_{x_k}(t \cdot \eta_k^0)) \geq t\inner{-\grad f(x_k)}{\eta_k^0} - \frac{Lt^2}{2} \|\eta_k^0\|^2.
	\label{eq:armijofoo}
	\end{align}
	We determine a sufficient condition on $t$ for the stopping criterion in Algorithm~\ref{algo:armijo} to trigger. To this end, observe that the right hand side of~\eqref{eq:armijofoo} dominates $c_1 t \inner{-\grad f(x_k)}{\eta_k^0}$ if
	$$
	t(1-c_1) \cdot \inner{-\grad f(x_k)}{\eta_k^0} \geq  \frac{Lt^2}{2} \|\eta_k^0\|^2.
	$$
	Thus, the stopping criterion in Algorithm~\ref{algo:armijo} is satisfied in particular for all $t$ in
	$$
	\left[0, \frac{2(1-c_1)\inner{-\grad f(x_k)}{\eta_k^0}}{\L\|\eta_k^0\|^2}\right] \supseteq \left[ 0, \frac{2c_2 (1-c_1)\|\grad f(x_k)\|}{\L\|\eta_k^0\|}\right] \supseteq \left[ 0, \frac{2c_2 (1-c_1)}{c_4\L}\right].
	$$
	Unless it equals ${\bar t}_k$, the returned $t$ cannot be smaller than $\tau$ times the last upper bound. In all cases, the cost decrease satisfies
	\begin{align*}
	f(x_k) - f(\Retr_{x_k}(t \cdot \eta_k^0)) & \geq c_1 t \inner{-\grad f(x_k)}{\eta_k^0} \\
	& \geq c_1c_2t\|\grad f(x_k)\|\|\eta_k^0\| \\
	& \geq c_1c_2c_3t\|\grad f(x_k)\|^2.
	\end{align*}
	To count the number of iterations, consider that checking whether $t = {\bar t}_k$ satisfies the stopping criterion requires one cost evaluation. Following that, $t$ is reduced by a factor $\tau$ exactly $\log_\tau(t/{\bar t}_k) = \log_{\tau^{-1}}({\bar t}_k/t)$ times, each followed by one cost evaluation.
\end{proof}

\section{Proofs for Section~\ref{sec:HandHess} about $H_k$ and the Hessians}\label{apdx:Hfirstretr}

\begin{proof}[\protect{Proof of Lemma~\ref{lem:scndretr}}]
The Hessian of $f$ and that of the pullback are related by the following formulas. See~\citep[\S5]{AMS08} for the precise meanings of the differential operators~$\mathrm{D}$ and~$\mathrm{d}$. For all $\eta$ in $\T_{x}\calM$, writing $\hat f_x = f\circ\Retr_x$ for convenience,
\begin{align*}
\frac{\mathrm{d}}{\mathrm{d}t} f(\Retr_{x}(t\eta)) & = \inner{\grad f(\Retr_{x}(t\eta))}{\frac{\mathrm{D}}{\mathrm{d}t} \Retr_x(t\eta)}, \\
\inner{\nabla^2 \hat f_x(0_{x})[\eta]}{\eta} & = \left. \frac{\mathrm{d}^2}{\mathrm{d}t^2} f(\Retr_{x}(t\eta)) \right|_{t=0} \\
& = \inner{\Hess f(x)\left[\mathrm{D}\Retr_x(0_x)[\eta]\right]}{\left. \frac{\mathrm{D}}{\mathrm{d}t} \Retr_x(t\eta) \right|_{t=0}} \\
& \quad + \inner{\grad f(x)}{\left. \frac{\mathrm{D}^2}{\mathrm{d}t^2} \Retr_x(t\eta) \right|_{t=0}} \\
& = \inner{\Hess f(x)[\eta]}{\eta} + \inner{\grad f(x)}{\left. \frac{\mathrm{D}^2}{\mathrm{d}t^2} \Retr_x(t\eta) \right|_{t=0}}.
\end{align*}
(To get the third equality, it is assumed one is working with the Levi--Civita connection, so that $\Hess f$ is indeed the Riemannian Hessian.) 
Since the acceleration of the retraction is bounded, we get the result via Cauchy--Schwarz.
\end{proof}

\begin{proof}[\protect{Proof of Proposition~\ref{prop:relateHandHess}}]
	Combine $\|\grad f(x_k)\| \leq \varepsilon_g$ and $H_k \succeq -\varepsilon_H \Id$ with
	\begin{align*}
	\left\| \Hess f(x_k) - \nabla^2 \hat f_{x_k}(0_{x_k}) \right\| & \leq a_k \cdot \|\grad f(x_k)\| & \textrm{ and } & & \left\|\nabla^2 \hat f_k(0_{x_k}) - H_k\right\| & \leq \delta_k
	\end{align*}
	by triangular inequality.
\end{proof}

\section{Complexity dependence on $n$ in the Max-Cut example}
\label{apdx:examplecomplexity}

This appendix supports Section~\ref{sec:example}. By Proposition~\ref{prop:maxcutapprox}, running Algorithm~\ref{algo:RTRscnd} with $\varepsilon_g = \infty$ and $\varepsilon_H = \frac{2\delta}{n}$ yields a solution $Y$ within a gap $\delta$ from the optimal value of~\eqref{eq:BM}. Let $\underline{f}$ and $\overline{f}$ denote the minimal and maximal values of $f(Y) = \inner{C}{YY\transpose}$ over $\calM$~\eqref{eq:calMmaxcut}, respectively, with metric $\inner{A}{B} = \trace(A\transpose B)$ and associated Frobenius norm $\frobnormsmall{\cdot}$. Then, using $\rho' = 1/10$, setting $c_3 = 1/2$ in~\aref{assu:rtr-eigenstep} as allowed by Lemma~\ref{lem:EigenStep} and using the true Hessian of the pullbacks for $H_k$ so that $c_1 = 0$ in~\aref{assu:rtr-Hk-second}, Theorem~\ref{thm:rtr-scnd-main} guarantees Algorithm~\ref{algo:RTRscnd} returns an answer in at most
\begin{align}
	214(\overline{f} - \underline{f}) \cdot \LH^2 \cdot \frac{1}{\varepsilon_H^3} \textrm{ + log term}
\end{align}
iterations. Using the $LDL\transpose$--factorization strategy of Lemma~\ref{lem:EigenStep} with a randomly generated orthonormal basis at each tangent space encountered, since $\dim \calM = n^2$ for $p = n+1$, the cost of each iteration is $\calO(n^6)$ arithmetic operations (dominated by the cost of the $LDL\transpose$ factorization). It remains to bound $\LH$, in compliance with~\aref{assu:rtr-regularity-second}.

Let $g \colon \reals \to \reals$ be defined as $g(t) = f(\Retr_Y(t\dot Y))$. Then, using a Taylor expansion,
\begin{align}
	f(\Retr_Y(\dot Y)) = g(1) = g(0) + g'(0) + \frac{1}{2} g''(0) + \frac{1}{6} g'''(t)
\end{align}
for some $t\in(0, 1)$.
Let $\hat f_Y = f \circ \Retr_Y$. Definition~\ref{def:retraction} for retractions implies
\begin{align}
	g(0) & = f(Y), & g'(0) & = \inner{\grad f(Y)}{\dot Y}, & g''(0) & = \inner{\dot Y}{\nabla^2 \hat f_Y(0_Y)[\dot Y]}, \end{align}
so that it only remains to bound $|g'''(t)|$ uniformly over $Y, \dot Y$ and $t \in [0, 1]$.

For this example, it is easier to handle $g'''$ if the retraction used is the exponential map (similar bounds can be obtained with the orthogonal projection retraction, see~\citep[Lemmas~4 and 5]{mei2017solvingSDPs}). This map is known in explicit form and is cheap to compute for the sphere $\mathbb{S}^{n} = \{ x \in \reals^{n+1} : x\transpose x = 1 \}$. Indeed, if $x \in \mathbb{S}^{n}$ and $\eta \in \T_x\mathbb{S}^n$, following~\citep[Ex.~5.4.1]{AMS08},
\begin{align}
	\gamma(t) = \Exp_x(t\eta) & = \cos(t\|\eta\|) x + \sin(t\|\eta\|) \frac{1}{\|\eta\|} \eta.
\end{align}
Conceiving of $\gamma$ as a map from $\reals$ to $\reals^{n+1}$, its differentials are easily derived:
\begin{align}
	\dot \gamma(t) & = -\|\eta\|\sin(t\|\eta\|) x + \cos(t\|\eta\|) \eta, & \ddot \gamma(t) & = - \|\eta\|^2 \gamma(t), & \dddot \gamma(t) & = - \|\eta\|^2 \dot \gamma(t).
\end{align}
Extending this map row-wise gives the exponential map for~$\calM$---of course, this is a second-order retraction. We define $\Phi(t) = \Retr_Y(t\dot Y)$ and $g(t) = f(\Retr_Y(t \dot Y)) = \inner{C\Phi(t)}{\Phi(t)}$. In particular, $\ddot \Phi(t) = -D\Phi(t)$ and $\dddot \Phi(t) = -D\dot\Phi(t)$, where $D = \diag(\|\dot y_1\|^2, \ldots, \|\dot y_n\|^2)$ and $\dot y_k\transpose$ is the $k$th row of $\dot Y$. As a result, for a given $Y$ and $\dot Y$, a little bit of calculus gives:
\begin{align}
	g'''(t) & = -6\inner{C\dot\Phi(t)}{D\Phi(t)} - 2\inner{C\Phi(t)}{D\dot\Phi(t)}.
\end{align}
Using Cauchy--Schwarz multiple times, as well as the inequality $\frobnormsmall{AB} \leq \opnorm{A}\frobnormsmall{B}$ where $\opnorm{A}$ denotes the largest singular value of $A$, and using that $\frobnormsmall{\Phi(t)} = \sqrt{n}$ and $\frobnormsmall{\dot \Phi(t)} = \frobnormsmall{\dot Y}$ for all $t$, and additionally that $\opnorm{D} \leq \trace(D) = \frobnormsmall{\dot Y}^2$, it follows that
\begin{align}
	\sup_{Y\in\calM, \dot{Y}\in\T_Y\calM, \dot Y\neq 0, t\in(0, 1)} \frac{|g'''(t)|}{\frobnormsmall{\dot Y}^3} \leq 8\opnorm{C}\sqrt{n}.
\end{align}
As a result, an acceptable constant $\LH$ for~\aref{assu:rtr-regularity-second} is $\LH = 8\opnorm{C} \sqrt{n}.$ 

Combining all statements of this section, it follows that a solution $Y$ within an absolute gap $\delta$ of the optimal value can be obtained for problem~\eqref{eq:BM} using Algorithm~\ref{algo:RTRscnd} in at most $\calO\left( (\overline{f} - \underline{f}) \opnorm{C}^2 \cdot n^{10} \cdot \frac{1}{\delta^3} \right)$ arithmetic operations, neglecting the additive logarithmic term.

Note that, following~\citep[Appendix~A.2, points 1 and 2]{mei2017solvingSDPs}, it is also possible to bound $\LH$ as $6\opnorm{C} + 2\|C\|_1$, where $\|\cdot\|_1$ is the $\ell_1$ operator norm. This reduces the explicit dependence on $n$ from $n^{10}$ to $n^9$ in the bound on the total amount of work.

\end{document}